\documentclass[12pt]{article}
\usepackage{amssymb}
\usepackage{amsmath}
\usepackage[all,arc]{xy}
\usepackage{ytableau}
\usepackage{hyperref}
\usepackage{tikz}
\usepackage[margin=1in]{geometry}

\usepackage{lscape}

\usepackage{amsthm}
\theoremstyle{definition}\newtheorem{remark}{Remark}[section]
\theoremstyle{definition}
\theoremstyle{plain}\newtheorem{propo}[remark]{Proposition}	
\newtheorem{theorem}[remark]{Theorem}
\newtheorem{example}[remark]{Example}
\newtheorem{definition}[remark]{Definition}
\newtheorem{lemma}[remark]{Lemma}
\newtheorem{corollary}[remark]{Corollary}

\newtheorem{conjecture}[remark]{Conjecture}
\theoremstyle{definition}

\newcommand{\C}{\mathbb{C}}

\newcommand{\Z}{\mathbb{Z}}
\newcommand{\N}{\mathbb{N}}

\ytableausetup{smalltableaux}

\begin{document}
	
	%	\begin{frontmatter}
	
	\title{Twisted Affine Lie Algebras, Fusion Algebras, and Congruence Subgroups}
	
	\author{Alejandro Ginory}
	\date{}
	\maketitle	
	%\address{}
	
\begin{abstract}
	\noindent The space spanned by the characters of twisted affine Lie algebras admit the action of certain congruence subgroups of $SL(2,\Z)$. By embedding the characters in the space spanned by theta functions, we study an $SL(2,\Z)$-closure of the space of characters. Analogous to the untwisted affine Lie algebra case, we construct a commutative associative algebra (fusion algebra) structure on this space through the use of the Verlinde formula and study important quotients. Unlike the untwisted cases, some of these algebras and their quotients, which relate to the trace of diagram automorphisms on conformal blocks, have \emph{negative} structure constants with respect to the (usual) basis indexed by the dominant integral weights of the Lie algebra. We give positivity conjectures for the new structure constants and prove them in some illuminating cases. We then compute formulas for the action of congruence subgroups on these character spaces and give explicit descriptions of the quotients using the affine Weyl group.
\end{abstract}

\section{Introduction}
\label{introduction}
Fusion algebras associated with affine Lie algebras arise in the study of the Wess-Zumino-Novikov-Witten model, which is a two-dimensional rational conformal field theory [KZ]. This model features certain modules of so-called \emph{untwisted} affine Lie algebras at a fixed positive integer level (see the appendix \ref{affine-algebra-prelims} for definitions), and the corresponding fusion algebras give a way of `fusing' these modules together, akin to taking the traditional tensor products of vector spaces. E. Verlinde introduced the notion that the structure constants of the fusion algebras, called \emph{fusion rules}, are intimately connected with a certain action of the group $SL(2,\Z)$ on the space of characters of the relevant modules [V]. In fact, he conjectured a formula, called the \emph{Verlinde formula}, that explicitly computes the fusion rules in terms of the matrix entries of the action of a distinguished element $S\in SL(2,\Z)$ (corresponding to the M\"{o}bius transformation $\tau\mapsto -1/\tau$ on the upper half-plane $\mathbb{H}$).

Many proofs have been presented for Verlinde's conjecture, most notably Faltings [F], Teleman [T], and Huang [Hu1]. These results are from various points of view and have many applications. From the representation theory side, Huang constructed modular tensor categories (see [Tu] for the definition) whose Grothendieck rings are isomorphic to the fusion algebras [Hu2], giving a representation-theoretic interpretation of these algebras. (Our original motivation was to find an analogous interpretation in the twisted case.)

In this paper, we study analogous structures for \emph{twisted} affine Lie algebras and investigate their properties. Let us briefly summarize the contents. To begin with, we use the definition of twisted affine Lie algebras using generalized Cartan matrices, i.e., in terms of generators and relations (the generalized Chevalley-Serre relations). We then proceed to find the most suitable $SL(2,\Z)$-module in which to embed the space of characters of twisted affine algebras. This space is then endowed with the structure of a commutative associative algebra, which we refer to as a \emph{Verlinde algebra}, using the action of $S\in SL(2,\Z)$ in a fashion analogous to the construction of fusion algebras for untwisted affine algebras. For affine Lie algebras of type $A^{(2)}_{2\ell},$ the space of characters admits the action of $SL(2,\Z)$ so we define an algebra structure on it directly. In the remaining cases of twisted affine Lie algebras we consider a strictly larger space on which $SL(2,\Z)$ acts. 

Using the algebra structure and the affine Weyl group, we construct a natural family of quotients, in which, quite surprisingly, negative structure constants appear. These quotients are isomorphic to the fusion algebras constructed in [Ho2] which encode the trace of diagram automorphisms on conformal blocks. In [Ho2], the author discovers that the structure constants are non-negative in certain families and asks how generally this property holds. Our results show that the non-negativity of the structure constants holds in about `half' the cases (depending on level), while in the other `half' we determine when negative structure constants appear. Ultimately, we hope that the Verlinde algebras lead to a greater understanding of the structure of certain categories of modules for twisted affine Lie algebras.

Now we will describe the contents of this paper in more technical detail. The space of characters of an affine Lie algebra of type $X_N^{(r)}$ admits an action of the congruence subgroup $\Gamma_1(r)\subseteq SL(2,\Z)$. Untwisted affine algebras correspond to the case when $r=1$. Therefore, $\Gamma_1(1)=SL(2,\Z)$ acts on the space of characters of untwisted affine algebras and, using the action of the element $S=\begin{pmatrix}
0 & -1\\1& 0
\end{pmatrix}$, one can define a commutative associative algebra structure on this space explicitly using the so-called \textit{Verlinde formula}, which we will define shortly.

The construction of the algebras associated to twisted affine algebras is as follows. First, suppose that $V$ is a finite dimensional $SL(2,\Z)$-module with a fixed basis $\{e_{\lambda}\}_{\lambda\in I}$, where $0\in I$ is a distinguished index. If the 0th column of the action of $S\in SL(2,\Z)$, with respect to this basis, has no zero entries, then one can put a commutative associative algebra structure on $V,$ with $e_0$ as the identity element, by defining its structure constants as
\begin{equation}
\label{verlinde-formula-intro}
N_{\lambda\mu}^{\nu}=\sum_{\phi\in I}\frac{S_{\phi\lambda}S_{\phi\mu}(S^{-1})_{\nu\phi}}{S_{\phi 0}}
\end{equation}
(see [Ka, ex. 13.34]). In this paper, the above formula will be referred to as the \emph{Verlinde formula}. Notice that any scalar multiple of $S$ will yield the same structure constants $N_{\lambda\mu}^{\nu}$. 

We will now construct such a space associated to the character space of twisted affine algebras. Consider the space of characters $Ch_k$ of fixed positive integer level $k$ of an affine Lie algebra $\mathfrak{g}$. This complex vector space is a subspace of the space $Th_k$ of level $k$ theta functions and is spanned by $\chi_{\lambda}=A_{\lambda+\rho}/A_{\rho}$, where $A_{\mu}$ are alternating sums of theta functions (which we refer to as \emph{alternants}). The space spanned by $A_{\lambda+\rho}$, for $\lambda$ ranging over dominant level $k$ weights, is denoted $Ch_{k}^{-}$ and is a subspace of $Th_{k+h^\vee}$, where $h^\vee$ is the dual Coxeter number of $\mathfrak{g}$. In the untwisted cases and in the case that $\mathfrak{g}$ is of type $A^{(2)}_{2\ell},$ the map $\chi_{\lambda}\mapsto A_{\lambda+\rho}=\chi_{\lambda}A_{\rho}$ is an injection and intertwines the action of $SL(2,\Z)$ up to scalar. For the remaining twisted cases, $Ch_{k}^{-}$ does not admit the action of $SL(2,\Z)$, but its image under $S$ is naturally contained in a certain subspace $V\subseteq Th_{k+h^\vee}$. This subspace $V$ is the smallest subspace of $Th_{k+h^\vee}$ that contains $Ch_{k}^{-}$, is closed under $S$, and has a basis $\{A_{\mu}\}_{\mu\in X}$, where $X$ is a subset of the set of regular dominant weights (not necessarily integral of course). The space $V$ is called the $SL(2,\Z)$-\emph{closure} of $Ch_{k}^{-}.$ (The condition that $V$ be spanned by $A_{\mu}$ reflects the hope that this space will correspond to integrable highest weight modules for some affine Lie algebra and thus can be interpreted representation-theoretically.) More precisely, we have the following description of $V$.

\begin{theorem}
	\label{min-S-module}
	If $\mathfrak{g}$ is a twisted affine Lie algebra, but not of type $A^{(2)}_{2\ell}$, then the $SL(2,\Z)$-closure of $Ch_k^-(\mathfrak{g})$ is equal to $Ch_k^-(\mathfrak{g}^t),$ where $\mathfrak{g}^t$ is the transpose affine Lie algebra. In the remaining cases, $Ch_k^-(\mathfrak{g})$ is an $SL(2,\Z)$-module.
\end{theorem}

If $\mathfrak{g}$ is untwisted, then the map $\chi_{\lambda}\mapsto A_{\lambda+\rho}$ is an algebra isomorphism (since the Verlinde formula only depends on the action of $S$ up to scalar). If $\mathfrak{g}$ is of type $A^{(2)}_{2\ell},$ we get a completely new algebra defined on the space of characters. For the remaining twisted cases, we get an algebraic structure on a strictly larger space (the precise increase in size will be discussed in detail in section \ref{r>a_0 case}). In both cases, the identity element is chosen to be the alternant corresponding to the `smallest' weight. Our discussion motivates the following unified definition.

\begin{definition}
	\label{twisted-verlinde-definition}
	Let $\mathfrak{g}$ be an affine Kac-Moody Lie algebra and $k$ be a positive integer. The \emph{Verlinde algebra} $V_k(\mathfrak{g})$ is the algebra whose underlying vector space is the $SL(2,\Z)$-closure $V$ with basis $\{A_{\lambda}\}_{\lambda\in X},$ where $X$ is a subset of regular dominant weights, (see the discussion above \ref{min-S-module}) and whose structure constants, with respect to the given basis, are given by the Verlinde formula (eq. \ref{verlinde-formula-intro}).
\end{definition}

This definition has several advantageous features. The first is that the action of $\Gamma_1(r)$ on the space of characters can be exhibited directly by the $SL(2,\Z)$ action on the untwisted Verlinde algebra. Indeed, explicit formulas for the action of congruence subgroups on the linear space of characters of twisted affine algebras are given in terms of evaluations of characters of irreducible finite dimensional modules for finite dimensional simple Lie algebras similar to the untwisted case. The second is that the algebra has structure constants that can be computed via the Verlinde formula. Another, as we shall later see, is that the Verlinde algebras can be defined as quotients of the representation ring $\mathcal{R}(\overset{\circ}{\mathfrak{g}}{}^t)$ of the transpose of the underlying simple Lie algebra (the transpose Lie algebra of $B_n$ is $C_n$, $C_n$ is $B_n$, and the remaining are self-transpose).

In all cases except for $A_{2\ell}^{(2)}$, these Verlinde algebras turn out to be isomorphic to the Verlinde algebras of untwisted affine algebras at \emph{different levels}. By construction, this isomorphism is compatible with the $SL(2,\Z)$ action. More specifically, the following theorem is proved. (Note that in the case of $A^{(2)}_{2\ell}$ this is a tautology since these algebras are isomorphic to their transposes.) 

\begin{theorem}
	\label{twisted-untwisted-isomorphism}
	Let $\mathfrak{g}$ be a twisted affine Lie algebra of type $X_N^{(r)}$, $k$ a positive integer, $\mathfrak{g}^t$ its transpose algebra, then there is an algebra isomorphism 
	$$\varphi_k:V_k(\mathfrak{g})\to V_{k+h^\vee -h}(\mathfrak{g}^t)$$
	that intertwines the action of $SL(2,\Z)$ up to scalar. 
\end{theorem}

Verlinde algebras have striking features, including the following algebra grading structure whose proof relies on the Kac-Walton algorithm (see section \ref{r=1 section}). 

\begin{theorem}
	\label{Verlinde-grading}
	Let $\mathfrak{g}$ be an affine Lie algebra of type $X_{N}^{(1)},$ $V$ be its Verlinde algebra at level $k$, and set $\mathcal{A}=\overset{\circ}{P}/\overset{\circ}{Q}$. The algebra $V$ has an $\mathcal{A}$-grading
	$$V=\bigoplus_{a\in \mathcal{A}}V_a$$
	where $V_a=\text{span}\{\chi_{\lambda}:\lambda\in P_k^+\text{ and } \overline{\lambda}\in a\}.$
\end{theorem}

In certain cases, the grading structure highlights the subspace of characters of twisted affine algebras, but not in general. The grading is not immediately evident from studying the $S$-matrix but can be observed after computing the structure constants. The grading seems to play a significant role in the Verlinde algebras of type $A^{(2)}_{2\ell}$ and in certain quotients of the other types (with multiple root lengths).

In all but the case when the affine Lie algebra is of type $A_{2\ell}^{(2)}$ with even level, the structure constants in these Verlinde algebras are known to be non-negative integers. In the $A_{2\ell}^{(2)}$ even level case, all known examples have some negative integral structure constants. We prove that in the special case $A^{(2)}_2,$ at even levels, the Verlinde algebra always has some negative structure constants and the grading helps us determine the signs of all structure constants. The author has not found this phenomenon cited in any of the literature. Based on a large amount of evidence, we conjecture (see \ref{r=a_0=2 case}) that `twisting' the basis by a character on $\overset{\circ}{P}/\overset{\circ}{Q}$ gives nonnegative integral structure constants in type $A^{(2)}_2.$\\
\\
\noindent \textbf{Remark}. There is some difficulty in working with this definition of the algebra since computing the structure constants is quite tedious if done `by hand', involving large sums of roots of unity. We overcame this problem by writing and using the Maple package {\tt Affine Lie Algebras}, which can efficiently compute these quantities and helped reveal much of the phenomena studied here. This program helped immensely in testing our conjectures for many families of affine Lie algebras.

The same algebraic structure can be defined in terms of a certain action of the affine Weyl group on the representation ring of an associated finite dimensional simple Lie algebra. While this sidesteps the problem of the aforementioned computations and is more conceptual, working with the algebraic structure in this setting involves finding specific representatives in the infinite orbits of the affine Weyl group, again a difficult task. Nonetheless, in some new cases arising from our methods, we present a recursive method of computing these representatives by finding a natural family of affine Weyl group elements that move the vectors they act on closer to a certain fundamental domain of the action of the group (see section \ref{quotient-a2even}).

\subsection{Structure of Paper}

In section \ref{r=1 section}, the well-known Verlinde algebras for untwisted affine Lie algebras are reviewed and their structure is discussed. In section \ref{r>a_0 case}, the definition of Verlinde algebras is given for all twisted affine Lie algebras. The grading structure and linear subspace spanned by the twisted affine algebra characters are discussed. In section \ref{r=a_0=2 case}, the Verlinde algebras of type $A_{2\ell}^{(2)}$ are more extensively studied. When the level is odd, these algebras are well understood and are isomorphic to an untwisted Verlinde algebra. When the level is even, as mentioned earlier, the structure constants are sometimes negative. Finally, in section \ref{applications section}, we show that J. Hong's fusion algebras for affine Kac-Moody Lie algebras are isomorphic to natural quotients of the Verlinde algebras presented here and thereby answer some open questions about non-negative structure constants posed in [Ho2]. The quotient structure of the Verlinde algebra of type $A^{(2)}_{2\ell}$ is then explicitly computed in terms of the affine Weyl group. In the process, we produce an algorithm for bringing weights that lie in $2kA_1$ into $kA_{1}$, where $A_1$ is the fundamental alcove of the affine Weyl group of type $C$ and $k$ is a positive integer. The algorithm is a technical tool can be used to study the negative structure constants but is of interest in its own right. In addition, explicit formulas for the action of congruence subgroups on the space of characters are given, proving certain identities such as the following one (which is very similar to the classical $S$-matrix action).

$$\sum_{\nu\in \overset{\circ}{P}}e^{-2\pi i m_{\nu}}(S_{k\Lambda_0,\nu})^2\overset{\circ}{\chi}_{\lambda}\left(e^{-\frac{2\pi i (\nu+\rho)}{k+h^\vee}}\right)\overset{\circ}{\chi}_{\mu}\left(e^{\frac{2\pi i (\nu+\rho)}{k+h^\vee}}\right)=e^{2\pi i (m_{\lambda}+m_{\mu})}S_{k\Lambda_0,\mu}\overset{\circ}{\chi}_{\lambda}\left(e^{-\frac{2\pi i (\mu+\rho)}{k+h^\vee}}\right)$$

\noindent The action of the element $u_{21}^r$ of $\Gamma_1(r)$ is computed and shown to be expressible as multiples of evaluations of \emph{summands} of the characters of finite dimensional simple Lie algebras. In the appendix \ref{notations-conventions}, we introduce notation and briefly go over the basic representation theory of affine Kac-Moody Lie algebras. 

\bigskip

\subsection{Acknowledgements}

I would like to thank Siddhartha Sahi for introducing me to this subject and for the many fruitful discussions regarding it. I am also grateful to Angela Gibney, Daniel Krashen, and James Lepowsky for their insightful and useful comments on an early draft of this paper. In addition, I would like to acknowledge Jason Saied and Eric Wawerczyk for excellent feedback.

\section{Verlinde Algebras of Untwisted Affine Algebras}
\label{r=1 section}

This section introduces the well-known Verlinde algebras of untwisted affine Lie algebras, i.e., affine Kac-Moody Lie algebras of type $X_N^{(1)}$, and discuss their grading structure, including the proof of Theorem \ref{Verlinde-grading}.

Let $\mathfrak{g}$ be an untwisted affine Lie algebra and $k$ a positive integer that will serve as the level. Recall that in this case $Ch_k$ is invariant under the action of $SL(2,\Z)$. More importantly, the action of the element $S\in SL(2,\Z)$ (defined in \ref{theta-fns-characters}) encodes information about a commutative associative algebra structure on $Ch_k$. The $S$-\emph{matrix}, which gives the action of $S$ with respect to the basis of characters, encodes the structure constants of this algebra via the Verlinde formula, which we restate for this particular case,
$\chi_{\lambda}\chi_{\mu}=\sum_{\nu\in P_k^+\text{ mod }\C\delta}N_{\lambda\mu}^{\nu}\chi_{\nu}$ where
$$N_{\lambda\mu}^{\nu}=\sum_{\phi\in P_k^{+}\text{ mod }\C\delta}\frac{S_{\lambda\phi}S_{\mu\phi}(S^{-1})_{\nu\phi}}{S_{k\Lambda_0,\phi}}.$$ 
The structure constants $N_{\lambda\mu}^{\nu}$ are known as the \emph{fusion rules}, are well-defined, and correspond to a tensor product structure on an appropriate category of modules for $\mathfrak{g}$ (which will not be discussed here). 

First, let us give an explicit description of the $SL(2,\Z)$ action on $Ch_k$, which can be found in [Ka]. Afterwards, it will be apparent that the Verlinde algebras are well-defined. Surprisingly, the $S$-matrix can be expressed in terms of evaluations of characters of irreducible highest weight $\overset{\circ}{\mathfrak{g}}$-modules. In what follows, the function $\overset{\circ}{\chi}_{\overline{\lambda}}$ is the character of the irreducible highest weight $\overset{\circ}{\mathfrak{g}}$-module with highest weight $\overline{\lambda}$. By [KP], $Ch_k$ is invariant under the action of $SL(2,\Z)$. The action is described explicitly in terms of the generators $T$ and $S$.

\bigskip

\begin{theorem} [Kac-Peterson] 
	\label{untwisted-modular-action}
	Let $\mathfrak{g}$ be an affine Lie algebra of type $X_{N}^{(1)}$ or $A_{2\ell}^{(2)}$ and $k$ be a positive integer. For any $\lambda\in P_k^+,$
	\begin{align*}
	\chi_{\lambda}|_T&=e^{2\pi i m_{\overline{\lambda}}}\chi_{\lambda}\\
	\chi_{\lambda}|_S&= \sum_{\mu\in P_k^+\text{ mod }\C\delta}S_{\lambda\mu}\chi_{\mu}
	\end{align*}
	where
	\begin{align*}
	S_{\lambda\mu}&=i^{|\overset{\circ}{\Delta}_+|}|M^*/(k+h^\vee)M|^{-1/2}\sum_{w\in \overset{\circ}{W}}\epsilon(w)e^{-\frac{2\pi i (\overline{\lambda}+\overline{\rho},w(\overline{\mu}+\overline{\rho}))}{k+h^\vee}}\\
	&=c_{\lambda}\overset{\circ}{\chi}_{\overline{\mu}}\left(e^{-\frac{2\pi i (\overline{\lambda}+\overline{\rho})}{k+h^\vee}}\right)
	\end{align*}
	and $c_{\lambda}=S_{k\Lambda_0,\lambda}=|M^*/(k+h^\vee)M|^{-1/2}\prod_{\alpha\in \overset{\circ}{\Delta}_+}2\sin\frac{\pi(\overline{\mu}+\overline{\rho},\alpha)}{k+h^\vee}$.
\end{theorem}

\bigskip

Note that the constant $c_\lambda=S_{k\Lambda_0,\lambda}$ in the above theorem is a positive real number. Furthermore, the theorem shows that the $S$-matrix is symmetric, therefore the product in $V_k$ is commutative. The associativity follows from the fact that $\chi_{k\Lambda_0}$ is decreed to be the identity element and the columns of the $S$-matrix are assumed to be, up to nonzero scalars, a system of orthogonal idempotents of $V_k$ and, under this assumption, the Verlinde formula simply computes the structure constants of the resulting algebra. It is worth noting that $S^{-1}=\overline{S}^t=\overline{S}$. We now proceed to prove Theorem \ref{Verlinde-grading}, i.e., that these Verlinde algebras are graded by $\overset{\circ}{P}/\overset{\circ}{Q}$.

\begin{proof}[Proof of \ref{Verlinde-grading}]
	Let us denote by $P(\overline{\mu})$ the set of weights (regarded as a multiset) of the irreducible finite dimensional $\overset{\circ}{\mathfrak{g}}$-module with highest weight $\overline{\mu}$. By the Kac-Walton algorithm [W], the fusion rules for $\lambda,\mu,\nu\in P_{k}^{+}$ are of the form
	$$N_{\lambda\mu}^{\nu}=\sum_{\phi}\epsilon(w_{\phi})N_{\overline{\lambda}\overline{\mu}}^{\overline{\phi}}$$
	where $\phi$ ranges over a finite multiset of integral weights of $\overset{\circ}{\mathfrak{g}}$ and $w_\phi$ are certain elements in $W$. Therefore, $N_{\overline{\lambda}\overline{\mu}}^{\overline{\nu}}=0$ implies that $N_{\lambda\mu}^{\nu}=0$. Using the Racah-Speiser algorithm [FH,\textsection 25.3], the finite fusion rules have the form
	$$N_{\overline{\lambda}\overline{\mu}}^{\overline{\nu}}=\sum_{\overline{\phi}}\epsilon(w_{\overline{\phi}})$$
	where the sum ranges over all $\overline{\phi}\in P(\overline{\mu})$, such that $w_{\overline{\phi}}(\overline{\lambda}+\overline{\phi}+\overline{\rho})-\overline{\rho}=\overline{\nu}$ for some $w_{\overline{\phi}}\in \overset{\circ}{W}$. Note that such $w_{\overline{\phi}}$ are unique (when they exist). 
	
	\bigskip
	
	In these cases, $M=\overset{\kern -.6em\circ}{Q^{\vee }}\subseteq \overset{\circ}{Q}\subseteq M^*=\overset{\circ}{P}$ so let $\mathcal{A}=M^{*}/\overset{\circ}{Q}$. Notice that $\overset{\circ}{W}$ acts trivially on $\mathcal{A}$. Define the subspaces
	$$V_a=\text{span}\{\chi_{\lambda}|\overline{\lambda}\equiv a\pmod{\overset{\circ}{Q}}\}$$
	for $a\in \mathcal{A}$. Recall that whenever $\overline{\phi}\in P(\overline{\mu}),$ $\overline{\mu}-\overline{\phi}$ is a sum of positive roots. Therefore, from the above discussion we know that
	$$\overline{\nu}=w_{\overline{\phi}}(\overline{\lambda}+\overline{\phi}+\overline{\rho})-\overline{\rho}\equiv \overline{\lambda}+\overline{\mu}\pmod{\overset{\circ}{Q}},$$
	proving that $N_{\lambda\mu}^{\nu}=0$ whenever $\overline{\lambda}+\overline{\mu}\neq \overline{\nu}$ in $\mathcal{A}$. This proves that $V=\bigoplus_{a\in \mathcal{A}}V_a$ is a grading.
\end{proof}

\bigskip

\begin{remark}
	The theorem above insinuates a relationship between the center of the associated compact Lie group $G$ and the grading structure. There may be an interpretation of this mystery in terms of an action of $Z(G)$ on the Verlinde algebra or on conformal blocks.
\end{remark}

\begin{example}
	Let us observe the grading structure in the case when $\mathfrak{g}$ is of type $A_2^{(1)}$ with level $k=2$. Consider the ordered basis $\chi_{\lambda_1},\chi_{\lambda_2},\ldots,\chi_{\lambda_6}$ where the $\overline{\lambda}_i$ are
	$$0,\overline{\Lambda}_1+\overline{\Lambda}_2,\overline{\Lambda}_1,2\overline{\Lambda}_2,2\overline{\Lambda}_1,\overline{\Lambda}_2,$$
	respectively. The weights have been paired off so that they lie in the classes $0,\overline{\Lambda}_1,2\overline{\Lambda}_1 \in \overset{\circ}{P}/\overset{\circ}{Q}$ respectively. The matrices for $L_{i}$, the operator of left multiplication by $\chi_{\lambda_i}$, are 
	$$L_1=\begin{pmatrix}
	1 & 0 &  &  &  &  \\
	0 & 1 &  &  &  &  \\
	&  & 1 & 0 &  &  \\
	&  & 0 & 1 &  &  \\
	&  &  &  & 1 & 0 \\
	&  &  &  & 0 & 1 
	\end{pmatrix} \quad L_2=\begin{pmatrix}
	0 & 1 &  &  &  &  \\
	1 & 1 &  &  &  &  \\
	&  & 1 & 1 &  &  \\
	&  & 1 & 0 &  &  \\
	&  &  &  & 0 & 1 \\
	&  &  &  & 1 & 1 
	\end{pmatrix}$$
	$$L_3=\begin{pmatrix}
	&  & & & 0 & 1 \\
	&  & & & 1 & 1 \\
	1  &1  &  & & &\\
	0  &1  &  & & &\\
	& &  1&  0&  &  \\
	& &  1&  1&  &  
	\end{pmatrix} \quad L_4=
	\begin{pmatrix}
	&  &  &  &  1& 0\\
	&  &  &  &  0&  1\\
	0&  1&  &  &  &  \\
	1&  0&  &  &  &  \\
	&  &  0&  1&  &  \\
	&  &  1&  0&  &  
	\end{pmatrix}$$
	$$L_5=
	\begin{pmatrix}
	&  &  0&  1&  &  \\
	&  &  1&  0&  &  \\
	&  &  &  &  0& 1 \\
	&  &  &  &  1& 0 \\
	1&  0&  &  &  &  \\
	0&  1&  &  &  &  
	\end{pmatrix} \quad L_6=
	\begin{pmatrix}
	&  &  1&  0&  &  \\
	&  &  1&  1&  &  \\
	&  &  &  &  1& 1 \\
	&  &  &  &  0& 1 \\
	0&  1&  &  &  &  \\
	1&  1&  &  &  &  
	\end{pmatrix}.$$
\end{example}

\begin{example}
	\label{B3-example}
	Let us observe the grading structure in the case when $\mathfrak{g}$ is of type $B_3^{(1)}$ with level $k=2$. Consider the ordered basis $\chi_{\lambda_1},\chi_{\lambda_2},\ldots,\chi_{\lambda_7}$ where the $\overline{\lambda}_i$ are
	$$0,\overline{\Lambda}_1,2\overline{\Lambda}_1,\overline{\Lambda}_2,2\overline{\Lambda}_3,$$
	$$\overline{\Lambda}_3,\overline{\Lambda}_1+\overline{\Lambda}_3$$
	respectively. There are only two classes $\{0,\overline{\Lambda}_3\}=\overset{\circ}{P}/\overset{\circ}{Q}$ and so $V_k$ has a $\Z_2$-grading. (Indeed, this is always the case for type $B_\ell^{(1)}$, with the system of representatives $\{0,\overline{\Lambda}_\ell\}$.) The first row lies in the class of $0$ and the second row lies in the class of $\overline{\Lambda}_3$. The matrices for the $L_{i}$, the operator of left multiplication by $\chi_{\lambda_i}$, are $L_1=id$
	$$L_2=\begin{pmatrix}
	0 & 1 & 0 & 0 & 0 &  &  \\
	1 & 0 & 1 & 1 & 0 &  & \\
	0 & 1 & 0 & 0 & 0 &  & \\
	0 & 1 & 0 & 0 & 1 &  & \\
	0 & 0 & 0 & 1 & 1 &  & \\
	&  &  &  &  & 1 & 1\\
	&  &  &  &  & 1 & 1
	\end{pmatrix} 
	\quad 
	L_3=
	\begin{pmatrix}
	0 & 0 & 1 & 0 & 0 &  &  \\
	0 & 1 & 0 & 0 & 0 &  & \\
	1 & 0 & 0 & 0 & 0 &  & \\
	0 & 0 & 0 & 1 & 0 &  & \\
	0 & 0 & 0 & 0 & 1 &  & \\
	&  &  &  &  &  0 & 1\\
	&  &  &  &  &  1 & 0
	\end{pmatrix}
	\quad 
	L_4=
	\begin{pmatrix}
	0 & 0 & 0 & 1 & 0 &  &  \\
	0 & 1 & 0 & 0 & 1 &  & \\
	0 & 0 & 0 & 1 & 0 &  & \\
	1 & 0 & 1 & 0 & 1 &  & \\
	0 & 1 & 0 & 1 & 0 &  & \\
	&  &  &  &  &  1 & 1\\
	&  &  &  &  &  1 & 1
	\end{pmatrix}$$
	$$L_5=
	\begin{pmatrix}
	0 & 0 & 0 & 0 & 1 &  &  \\
	0 & 0 & 0 & 1 & 1 &  & \\
	0 & 0 & 0 & 0 & 1 &  & \\
	0 & 1 & 0 & 1 & 0 &  & \\
	1 & 1 & 1 & 0 & 0 &  & \\
	&  &  &  &  &  1 & 1\\
	&  &  &  &  &  1 & 1
	\end{pmatrix}
	\quad L_6=
	\begin{pmatrix}
	&  &  &  &  &  1 & 0 \\
	&  &  &  &  &  1 & 1\\
	&  &  &  &  &  0 & 1\\
	&  &  &  &  &  1 & 1\\
	&  &  &  &  &  1 & 1\\
	1 & 1 & 0 & 1 & 1 &  & \\
	0 & 1 & 1 & 1 & 1 &  & 
	\end{pmatrix}
	\quad 
	L_7=
	\begin{pmatrix}
	&  &  &  &  &  0 & 1 \\
	&  &  &  &  &  1 & 1\\
	&  &  &  &  &  1 & 0\\
	&  &  &  &  &  1 & 1\\
	&  &  &  &  &  1 & 1\\
	0 & 1 & 1 & 1 & 1 &  & \\
	1 & 1 & 0 & 1 & 1 &  & 
	\end{pmatrix}.$$
\end{example}

\bigskip

This grading structure also applies to the case where $k=\infty$, i.e., when the Verlinde algebra is equal to the representation algebra of $\overset{\circ}{\mathfrak{g}}$, since it depends only on the Racah-Speiser algorithm for computing the tensor product decomposition of finite dimensional simple $\overset{\circ}{\mathfrak{g}}$-modules. 

\bigskip

A method of describing this algebra more conceptually is via the representation ring $\mathcal{R}(\overset{\circ}{\mathfrak{g}})$ of the finite dimensional simple Lie algebra $\overset{\circ}{\mathfrak{g}}$. (This interpretation of the Verlinde algebra will be used later on.) There is a map 
$$F_k:\mathcal{R}(\overset{\circ}{\mathfrak{g}})\to Ch_k(\mathfrak{g})$$
$$\overset{\circ}{\chi}_{\overline{\lambda}}\mapsto \left\{\begin{matrix}
\epsilon(w)\chi_{w(\lambda+\rho)-\rho} &\text{ if there exists }w\in W,\ w(\lambda+\rho)-\rho\in P_k^+\\
0 & \text{ otherwise}
\end{matrix}\right.$$
where $\lambda\in P_k^+$ (note that the $w$ referenced in the definition of $F_k$ is unique). There is also a canonical section 
$$S_k:Ch_k(\mathfrak{g})\to \mathcal{R}(\overset{\circ}{\mathfrak{g}})$$
$$\chi_{\lambda}\mapsto \overset{\circ}{\chi}_{\overline{\lambda}}.$$
The following proposition was proved in [W].

\bigskip

\begin{theorem}[Walton]
	The Verlinde algebra $V_k$ has the product structure 
	$$\chi_{\lambda}\cdot \chi_{\mu}=F_k(S_k(\chi_{\lambda})\otimes S_k(\chi_{\mu}))$$
	for $\lambda,\mu\in P_k^{+}$.
\end{theorem}

\bigskip

Instead of using $W$ in the definition of $F_k,$ we can use the affine Weyl group $W_{aff}^{k+h^\vee}=\overset{\circ}{W}\ltimes (k+h^\vee)\overset{\circ}{Q}{}^\vee$, which acts on $\overset{\circ}{\mathfrak{h}}{}^*$. The fundamental alcove $P^{aff\ +}_{k+h^\vee}$ is defined to be the set of $\lambda\in \overset{\circ}{\mathfrak{h}}{}^*$ such that
$$0\leq (\lambda,\alpha_i) \text{ for }i=1,\ldots,\ell$$
$$(\lambda,\theta)\leq k+h^\vee,$$
this is a fundamental domain for the action of $W_{aff}^{k+h^\vee}$. The set $P^{aff\ ++}_{k+h^\vee}$ consists of all $\lambda\in P^{aff\ +}_{k+h^\vee}$ such that the above inequalities are all strict. For any $\lambda\in \mathfrak{h}^*$ of level $k$, $\overline{\lambda}+k\Lambda_0\equiv \lambda\ (\text{mod }\C\delta)$. Using $\overline{\lambda}\mapsto k\Lambda_0+\overline{\lambda}$, the sets $P^{aff\ +}_{k}$ and $P^{aff\ ++}_{k}$ can be identified with $P^{+}_{k}$ and $P^{++}_{k}$, respectively. Therefore, we can also define the map $F_k$ by the rule
$$\overset{\circ}{\chi}_{\overline{\lambda}}\mapsto \left\{\begin{matrix}
\epsilon(w)\chi_{k\Lambda_0+w(\overline{\lambda}+\overline{\rho})-\overline{\rho}} &\text{ if there exists }w\in W_{aff}^{k+h^\vee},\ w(\overline{\lambda}+\overline{\rho})\in P^{aff\ ++}_{k+h^\vee}\\
0 & \text{ otherwise}
\end{matrix}\right. .$$
Later on in \ref{applications section}, this definition will prove very useful.

\bigskip

\begin{remark}
	While the representation ring description of these Verlinde algebras show that the structure constants are integers, these algebras have an interpretation as the Grothendieck ring of a tensor category of modules of untwisted affine Lie algebras and, therefore, the structure constants are all non-negative integers [Hu1]. For the twisted cases, the Verlinde algebras will be described as a quotient of the representation ring of a finite dimensional simple Lie algebra and non-negativity of structure constants will not hold in exactly the case $A_{2\ell}^{(2)}$ when the level is even.
\end{remark}

\section{Verlinde Algebras of Twisted Affine Algebras}
\label{r>a_0 case}

The affine Kac-Moody Lie algebras of type $X_N^{(r)}$ for $r>a_0$ are twisted affine Lie algebras and their spaces of characters behave very differently from those of untwisted affine Lie algebras. In these cases, $M=\overset{\circ}{Q}$ and $P^k\subsetneq P_k,$ therefore we cannot use the same techniques used for the algebras of type $X_N^{(1)}$. In particular, there is no action of the $S$-matrix. Nonetheless, to each twisted affine Lie algebra $\mathfrak{g}$ of this type, we will associate a Verlinde algebra $V_k(\mathfrak{g})$ that contains space of characters of $\mathfrak{g}$ at level $k$ and admits the action of the modular group. In particular, we prove theorems \ref{min-S-module} and \ref{twisted-untwisted-isomorphism} in this section by comparing a twisted affine algebra $\mathfrak{g}$ to its transpose $\mathfrak{g}^t$. Since the Verlinde algebras for type $A_{2\ell}^{(2)}$ are fundamentally different from the others, we will explore them more deeply in the next section (note that $r=a_0$ in this case).

\bigskip

The lattice $\overline{P}_k$ is the dual of the root lattice, therefore it has the basis
$$\frac{a_1}{a_1^\vee}\overline{\Lambda}_1, \frac{a_2}{a_2^\vee}\overline{\Lambda}_2,\ldots,\frac{a_\ell}{a_\ell^\vee}\overline{\Lambda}_\ell.$$ 
If $\lambda\in P_k^+$ then $(\overline{\lambda},\theta)=\langle\overline{\lambda},\theta^\vee\rangle\leq k.$
%$$\langle\lambda, K\rangle = k$$
%$$\langle\sum_{i=0}^{\ell}c_i\Lambda_i, \sum_{i=0}^{\ell}a_i^{\vee}\alpha_i^\vee\rangle = k$$
%$$\langle\sum_{i=1}^{\ell}c_i\Lambda_i, \sum_{i=1}^{\ell}a_i^{\vee}\alpha_i^\vee\rangle = k-c_0$$
Conversely, any $\overline{\lambda}\in \overline{P}^+$ satisfying the above inequality has unique preimage $\lambda \in P_k^+$ modulo $\C\delta.$

\bigskip

In [Ka], the author relates each $\mathfrak{g}$ of the type considered here to another affine algebra denoted $\mathfrak{g}'$ (this will be discussed later). For the purpose of investigating $Ch_k$, it will be helpful to instead relate $\mathfrak{g}$ to an affine Lie algebra of type $X_N^{(1)}$. To each affine Lie algebra with $r>a_0$, associate the transpose algebra $\mathfrak{g}^t=\mathfrak{g}(A^t)$, where $A$ is the Cartan matrix of $\mathfrak{g}$. The following table gives the explicit correspondences. 

%\begin{center}
%	\begin{tabular}{ c c }
%		$\mathfrak{g}$ & $\mathfrak{g}^t$\\ 
%		-----------&  -----------\\
%		$A^{(2)}_{2\ell-1}$ & $B^{(1)}_{\ell}$\\
%		$D^{(2)}_{\ell+1}$ & $C^{(1)}_{\ell}$\\
%		$E^{(2)}_{6}$ & $F^{(1)}_{6}$\\
%		$D^{(3)}_{2\ell-1}$ & $G^{(1)}_{\ell}$\\
%	\end{tabular}
%\end{center}

\begin{center}
	\begin{tabular}{||c |c|c||} 
		\hline
		$\mathfrak{g}$ & $\mathfrak{g}^t$ & $\mathfrak{g}'$  \\ 
		\hline\hline
		$A^{(2)}_{2\ell-1}$ & $B^{(1)}_{\ell}$ &$D^{(2)}_{\ell+1}$ \\ 
		\hline
		$D^{(2)}_{\ell+1}$ & $C^{(1)}_{\ell}$ & $A^{(2)}_{2\ell-1}$ \\ 
		\hline
		$E^{(2)}_{6}$ & $F^{(1)}_{6}$ & $E^{(2)}_{6}$\\ 
		\hline
		$D^{(3)}_{2\ell-1}$ & $G^{(1)}_{\ell}$ & $D^{(3)}_{2\ell-1}$\\ 
		\hline
	\end{tabular}
\end{center}

The superscript ${}^t$ will be used when dealing with data associated with $\mathfrak{g}^t$, e.g., the labels of $\mathfrak{g}^t$ are  denoted $a_i^t$, its Cartan subalgebra is denoted $\mathfrak{h}^t$, the invariant bilinear form is $(\cdot,\cdot)^t$, etc. In order to compare the two algebras more readily, define the isometry
$$\tau:\mathfrak{h}^*\to \mathfrak{h}^t$$
$$\alpha_i\mapsto \alpha_i^{\vee t}$$
$$\Lambda_0\mapsto d^t$$
which in turn induces an isometry $\mathfrak{h}^*\to (\mathfrak{h}^t)^{*}$, also denoted $\tau$, given by
$$\Lambda_i\mapsto \frac{a_i^{\vee}}{a_i}\Lambda_i^t \text{ and }\delta\mapsto \delta^t.$$
This map `transposes' the roots and coroots and is crucial in the construction of the Verlinde algebra for twisted affine algebras. 
%$$a_0\alpha_0+\cdots +a_\ell\alpha_\ell$$
%$$a_0\alpha_0^{\vee t}+\cdots +a_\ell\alpha_\ell^{\vee t}$$
%$$a_0^{\vee t}\alpha_0^{\vee t}+\cdots +a_\ell^{\vee t}\alpha_\ell^{\vee t}$$
The restriction map $\tau|_M$ is an isomorphism of lattices $M\cong M^t$ and 
$$\tau(\theta)=\theta^t, \tau(\theta^\vee)=\theta^{\vee t}.$$  The latter fact shows that $\tau\left(P^k\right)\hookrightarrow (P^k)^t=(P_k)^t\cong P_k$. Recall that $\{\Theta_{\lambda}|\lambda\in P_k\pmod{kM+\C\delta}\}$ is a basis of $Th_k$ for $k>0$. Therefore, $\tau$ induces a canonical map between the theta functions of $\mathfrak{g}$ and $\mathfrak{g}^t$.

\begin{propo}
	The isometry $\tau$ induces an $W$-invariant isomorphism $Th_k\cong Th_k^t.$
\end{propo}

\begin{proof}
	A direct check shows that $W^t = \overset{\circ}{W^t}\ltimes M^t \cong \overset{\circ}{W}\ltimes M$ and 
	$$e^{(\lambda,\mu)} = e^{(\tau(\lambda),\tau(\mu))^t}.$$
	Therefore, the map $\Theta_{\lambda}\mapsto \Theta^t_{\tau(\lambda)}$ is an inclusion. Furthermore, since $\tau(M^*)=(M^t)^*$, proposition \ref{theta-basis} implies that $\tau\left(Th_k\right)=Th_k^t$.
\end{proof}

%The cosets $P_{k}/P^{k}$ partition the basis of $Th_{k}$.
Since the finite Weyl groups of $\mathfrak{g}$ and $\mathfrak{g}^t$ are isomorphic and commute with isometries between their Cartan subalgebras, the space $Th_k^{-}$, which is spanned by the alternants 
$$A_{\lambda}=\sum_{w\in \overset{\circ}{W}}\epsilon(w)\Theta_{w(\lambda)}$$
for $\lambda\in P_k$, can be identified with the space $\left(Th_k^t\right)^{-}.$

\bigskip

At this point, one should notice that $\tau(\rho)\neq \rho^t$ and so, instead of using the map $\tau$ between theta functions given above, consider the map
$$\overset{\bullet}{\tau}:P_k\to P_{k+h^\vee-h}^t$$
$$\lambda\mapsto \tau(\lambda+\rho)-\rho^t$$
and the map it induces on the spaces of characters, denoted by the same symbol,
$$\overset{\bullet}{\tau}:Ch_{k}\to Ch_{k+h^\vee-h}^t$$
$$\chi_{\lambda}\mapsto \chi_{\overset{\bullet}{\tau}(\lambda)}^t.$$ 
Note that this map is a well-defined injection, i.e., $h^\vee-h\geq 0$ and $\overset{\bullet}{\tau}(\lambda)\in P_{k+h^\vee-h}^+$, and that $\overset{\bullet}{\tau}\left(\chi_{\lambda}\right)=A_{\tau(\lambda+\rho)}/A_{\rho^t}$. Furthermore, $\overset{\bullet}{\tau}$ gives a tight relationship between the $S$-matrices of the two Lie algebras. Before describing it, recall that for any $\mathfrak{g}$ of the type considered in this section, the \emph{adjacent algebra} $\mathfrak{g}'$ is given by the above table.
%\begin{center}
%	\begin{tabular}{ c c }
%		$\mathfrak{g}$ & $\mathfrak{g}'$\\ 
%		-----------&  -----------\\
%		$A^{(2)}_{2\ell-1}$ & $D^{(2)}_{\ell+1}$\\
%		$D^{(2)}_{\ell+1}$ & $A^{(2)}_{2\ell-1}$\\
%		$E^{(2)}_{6}$ & $E^{(2)}_{6}$\\
%		$D^{(3)}_{2\ell-1}$ & $D^{(3)}_{2\ell-1}$\\
%	\end{tabular}
%\end{center}
The $S$-matrix of $\mathfrak{g}$, as defined in [Ka], does not act on the space of characters but, rather, is a map between different character spaces. Nonetheless, its image can be described in terms of the characters of the adjacent algebra $\mathfrak{g}'$. For more details, please refer to [Ka,\textsection 13.9]. In that section, Kac discusses a map $\alpha\mapsto \alpha'$ between the Cartan subalgebras of $\mathfrak{g}$ and $\mathfrak{g}'$ that is very similar to the map $\tau$ in this paper.

\begin{propo}
	Let $S(\mathfrak{g})$ be the $S$-matrix acting on the space of theta functions at level $k$ and $S(\mathfrak{g}^t)$ be the analogous $S$-matrix for $\mathfrak{g}^t$ and level $k+h^\vee-h$. Then
	$$S(\mathfrak{g})_{\lambda,\mu} = \left|\overset{\kern -.5em\circ}{Q^\vee}/M\right|^{1/2}S(\mathfrak{g}^t)_{\overset{\bullet}{\tau}(\lambda),\tau(\mu)+(h^\vee-h)\Lambda_0^t}$$
	$\lambda\in P^k,$ $\mu\in P_k.$
\end{propo}

\begin{proof}
	The result follows directly from the dictionary between the map given in [Ka] and the construction of the two maps $\tau$ and $\overset{\bullet}{\tau}$ in this paper. Since $\overline{\rho}^t=\tau(\overline{\rho}')$ and $M'=\overset{\kern -.5em\circ}{Q^\vee}$, comparing the two quantities below 
	$$S(\mathfrak{g})_{\lambda,\mu}=i^{|\overset{\circ}{\Delta}_+|}|M^*/(k+h^\vee)M|^{-1/2}|M'/M|^{1/2}\sum_{w\in \overset{\circ}{W}}\epsilon(w)e^{-\frac{2\pi i (w(\overline{\lambda}+\overline{\rho}),\overline{\mu}+\overline{\rho}')}{k+h^\vee}}$$
	$$S(\mathfrak{g}^t)_{\overset{\bullet}{\tau}(\lambda),\tau(\mu)+(h^\vee-h)\Lambda_0^t}=i^{|\overset{\circ}{\Delta^t}_+|}|\left(M^t\right)^*/(k+h^\vee)M^t|^{-1/2}\sum_{w\in \overset{\circ}{W}}\epsilon(w)e^{-\frac{2\pi i (w(\tau(\overline{\lambda}+\overline{\rho})),\tau(\overline{\mu})+\overline{\rho}^t)}{k+h^\vee}}$$
	gives the result.
\end{proof}

\bigskip

Theorem \ref{min-S-module} asserts that the space $V$ is equal to $Th_{k+h^\vee}^{t-}$. Notice that $Ch_{k+h^\vee-h}^t\cong Th_{k+h^\vee}^{t-}$ via the map $\chi^t_{\lambda}\mapsto \chi^t_{\lambda}A_{\rho^t}$. Using this isomorphism, we can study the action of $S$ on $Th_{k+h^\vee}^{t-}$ and prove the theorems \ref{min-S-module} and \ref{twisted-untwisted-isomorphism}.

\bigskip

\begin{proof}[Proof of \ref{min-S-module}]
	The space $V$ is a subspace of $Th_{k+h^\vee}^{t-}$. The action of $S$ on $Th_{k+h^\vee}^{t-}$ differs from its action on $Ch_{k+h^\vee-h}^t$ by a nonzero scalar, see Theorem \ref{untwisted-modular-action}. The $S$-matrix of $Ch_{k+h^\vee-h}^t$ has $S_{k\Lambda_0,,\lambda}=S_{\lambda,k\Lambda_0}\neq 0$ for all $\lambda \in P_k^{t+}$ (see Theorem \ref{untwisted-modular-action}) and so the coefficient of $A_{\rho^t}$ in the expansion of $A_{\lambda+\rho}|_S$ (with resect to the alternant basis) is nonzero. It follows that $A_{\rho^t}\in V^-$. Similarly, for any $\mu\in P_k^{t+}$ the coefficient of $A_{\mu+\rho^t}$ in the expansion of $A_{\rho^t}|_S$ is nonzero. It follows that $Th_{k+h^\vee}^{t-}\subseteq V$.
	
	Note that $\widetilde{\rho}=\rho^t$ and that the $S$ matrix has nonzero entries in the column corresponding to $A_{\widetilde{\rho}}$. Therefore, we can use the Verlinde formula to give $V$ the structure of a commutative associative algebra as discussed in section \ref{introduction}.
\end{proof}

\bigskip

Theorem \ref{twisted-untwisted-isomorphism} is a direct consequence of above proof. Using the Verlinde algebra $V_{k+h^\vee-h}^t$ of $\mathfrak{g}^t$ at level $k+h^\vee-h$, one can study the action of the modular group on $Ch_k(\mathfrak{g})$ directly. 

\bigskip

\begin{theorem}
	\label{gamma1r-morphism}
	Let $\mathfrak{g}$ be a twisted affine Lie algebra, $\lambda \in P^k$, and $X\in \Gamma_1(r)$ then
	$$\tau\left(A_{\lambda+\rho}|_{X}\right)=A^t_{\tau(\lambda+\rho)}|_{X}.$$
	Furthermore, there is a nonzero scalar $v(k,X)$ such that
	$$\overset{\bullet}{\tau}\left(\chi_{\lambda}|_{X}\right)=v(k,X)\overset{\bullet}{\tau}(\chi_{\lambda})|_{X}.$$
\end{theorem}

\begin{proof}
	$\tau$ is an isometry on the underlying Cartan subalgebras and induces an isomorphism between the spaces of theta functions $Th_{k+h^\vee}$ and $Th_{k+h^\vee}^t$. Therefore, the modular group action commutes with $\tau$. The second identity follows from the fact that $A_{\rho}|_{X}\in \C A_{\rho}\setminus\{0\}$ and $A_{\rho^t}|_{X}\in \C A_{\rho^t}\setminus\{0\}$.
\end{proof}

\bigskip

Let us investigate the structure of $V_k$ relative $Ch_k$. The grading structure (introduced in \ref{Verlinde-grading}) on $V_k(\mathfrak{g})$, where $\mathfrak{g}$ is a twisted affine Lie algebras different from type $A_{2\ell}^{(2)}$, is easy to describe 
$$V_k(A^{(2)}_{2\ell-1})=V_0\oplus V_1,\ \ V_k(D^{(2)}_{\ell+1})=V_0\oplus V_1$$
$$V_k(E^{(2)}_{6})=V_0,\ \ V_k(D^{(3)}_{4})=V_0.$$
Using the grading and comparing it to the map $\overset{\bullet}{\tau}$, it can be shown that the image of $Ch_k(\mathfrak{g})$ lies in $V_{\tau(\overline{\rho})-\overline{\rho}^t}$ for $\mathfrak{g}$ of type $A_{2\ell-1}^{(2)}$. The image of $Ch_k$ for type $D^{(2)}_{\ell+1}$ is generally not concentrated in a homogeneous component.

\bigskip

\begin{propo}
	If $\mathfrak{g}$ is of type $A_{2\ell-1}^{(2)}$, then $V_k=V_0\oplus V_1$ and $\overset{\bullet}{\tau}\left(Ch_k\right)=V_1.$
\end{propo}

\begin{proof}
	The first part is clear. For the second part, note that
	$$\overset{\circ}{Q}{}^t(A_{2\ell-1}^{(2)})=\overset{\circ}{Q}(B_{\ell}^{(1)})=\langle \Lambda_1^t,\ldots,\Lambda_{\ell-1}^t,2\Lambda_{\ell}^t\rangle \text{ and }\tau(\rho)-\rho^t=\Lambda_{\ell}^t,$$
	where $\langle,\rangle$ means `spanned by'. Therefore, $V_{\tau(\overline{\rho})-\overline{\rho}^t}=V_1$. The lattice $\overset{\kern -.2em\circ}{Q^t}(A_{2\ell-1}^{(2)})$ is self-dual which implies the equality
	$$\overset{\bullet}{\tau}\left(Ch_k\right)= V_{\tau(\overline{\rho})-\overline{\rho}^t}^t=V_1 .$$
\end{proof}

\section{Verlinde Algebras for $A^{(2)}_{2\ell}$}
\label{r=a_0=2 case}

The case when $\mathfrak{g}$ is of type $A^{(2)}_{2\ell}$ has fascinating features that differ substantially from the other cases. On one hand, at \emph{odd} integer level $k$ the space of characters has the structure of a fusion algebra that is isomorphic to the Verlinde algebra for $C_{\ell}^{(1)}$ at the level $\frac{k-1}{2}$ [Ho2]. (A proof is included here.) On the other hand, more interesting features reveal themselves when the level is an \emph{even} integer. Among them is the fact that the structure constants are no longer non-negative integers.

\bigskip

With definition \ref{twisted-verlinde-definition}, it is easy to check that the fusion rules are well-defined and define a commutative associative unital algebra with integral structure constants. Similar to the untwisted Verlinde algebras, one can define an algebraic structure on the representation ring of $\mathfrak{g}$ via a map from the representation ring of $\overset{\circ}{\mathfrak{g}}$. Unlike the untwisted case, the fusion rules can only be shown to be non-negative at odd level.

\begin{propo}
	\label{A-2-even-odd-level-iso}
	Let $\mathfrak{g}$ be of type $A^{(2)}_{2\ell}$ and fix the level $2k+1$, where $k\in \N,$ then the Verlinde algebra of $\mathfrak{g}$ is isomorphic to the Verlinde algebra of the untwisted affine Lie algebra of type $C_{\ell}^{(1)}$ at level $k$.
\end{propo}

\begin{proof}
	It can be checked directly that the $S$-matrices of both affine Lie algebras are equal, up to conjugation by a permutation matrix. More explicitly, if $\mathfrak{g}'$ is of type $C_{\ell}^{(1)}$ with Cartan subalgebra $\mathfrak{h}'$, 
	$$2(\alpha_i^\vee,\alpha_j^\vee)=(\alpha_i^{\vee'},\alpha_j^{\vee'})'$$
	$$\frac{1}{2}(\overline{\Lambda}_i,\overline{\Lambda}_j)=(\overline{\Lambda}_i',\overline{\Lambda}_j')'$$
	for $1\leq i,j\leq \ell$. Therefore, we can define isometries
	$$\varphi:\overset{\circ}{\mathfrak{h}'}\to \overset{\circ}{\mathfrak{h}}$$
	$$\alpha\mapsto \sqrt{2}\alpha$$
	and the dual map
	$$\varphi^*:\overset{\circ}{\mathfrak{h}^{'*}}\to \overset{\circ}{\mathfrak{h}^*}$$
	$$\alpha\mapsto \frac{1}{\sqrt{2}}\alpha$$
	such that $\langle \varphi(\alpha),\varphi^*(\lambda)\rangle=\langle \alpha,\lambda\rangle$. Note that $\sqrt{2}M=\varphi(M')$ and $\frac{1}{\sqrt{2}}M^*=\varphi^*(M'^{*})$, therefore 
	$$|M'^*/(k+\ell+1)M'|=|M^*/(2k+2\ell+2)M|.$$
	The $S$-matrix of $\mathfrak{g}'$ at level $k$ is defined by
	$$S_{\lambda\mu}=i^{|\overset{\circ}{\Delta}_+|}|M'^*/(k+\ell+1)M'|^{-1/2}\sum_{w\in \overset{\circ}{W}}\epsilon(w)e^{-\frac{2\pi i (w(\overline{\lambda}+\overline{\rho}),\overline{\mu}+\overline{\rho})'}{k+\ell+1}}$$
	while the $S$-matrix of $\mathfrak{g}$ at level $2k+1$ is defined by
	$$S_{\lambda\mu}=i^{|\overset{\circ}{\Delta}_+|}|M^*/(2k+2\ell+2)M|^{-1/2}\sum_{w\in \overset{\circ}{W}}\epsilon(w)e^{-\frac{2\pi i (w(\overline{\lambda}+\overline{\rho}),\overline{\mu}+\overline{\rho})}{2k+2\ell+2}}.$$
	Since the structure constants of the Verlinde algebras are computed by the coefficients of the corresponding $S$-matrices, the algebras are isomorphic.
\end{proof}

In the cases where the level is even, one can observe that the fusion rules are often negative. Here is an example of this phenomenon. Let $\mathfrak{g}$ be of type $A_2^{(2)}$ and fix a level $2n,$ $n>0$. The underlying simple Lie algebra is isomorphic to $\mathfrak{sl}_2$. Set $\alpha=\alpha_1,$ $\alpha^\vee=\alpha_1^\vee, \overline{\Lambda}_1=\Lambda$ then
$$\theta=\alpha,\ \theta^\vee = \alpha^\vee,\ \text{and }\overline{\rho}=\alpha^\vee.$$
Note also that $h^\vee=3.$ The of dominant integral weights of level $2n$ project to
$$\overline{P}^{2n}_+=\{m\Lambda\}_{m=0}^{n}.$$
The set of dominant integral weights $P_+$ of $\mathfrak{sl}_2$ is indexed by $\N$ and the map between character spaces is
$$F_{2n}:\{\chi_m\}_{m\in P_+}\to \{\chi_{\lambda}\}_{\lambda\in P_+^{2n}}$$
$$\chi_{m}\mapsto \left\{\begin{matrix}
\chi_{m'\Lambda} & \text{ if }0\leq m'\leq n \text{ and }m\equiv m'\pmod{2n+h^\vee}\\
0 & \text{if }m\equiv -1\pmod{2n+h^\vee}\\
-\chi_{m'\Lambda} & \text{ if }0\leq m'\leq n \text{ and }m+m'+2\equiv 0\pmod{2n+h^\vee}
\end{matrix}\right. .$$
The proof of the above map follows from the fact that the affine Weyl group of type $A^{(2)}_{2\ell}$ is $S_2\ltimes \Z$ and so, for any $m\in \N$, there is a unique $m'$ modulo $2n+h^\vee$ satisfying the conditions in the piecewise definition of $F_{2n}$. 
More importantly, there is a canonical section 
$$S_{2n}:\{\chi_{\lambda}\}_{\lambda\in P_+^{2n}}\to\{\chi_{m}\}_{m\in P_+}$$
$$\chi_{m\Lambda}\mapsto \chi_{m}$$
and the fusion rules $N_{ab}^{c}$ are given by the coefficients of the irreducible character $\chi_{c}$ in
$$F_{2n}(S_{2n}(\chi_{a\Lambda})\otimes S_{2n}(\chi_{b\Lambda})).$$
Th product $S_{2n}(\chi_{a\Lambda})\otimes S_{2n}(\chi_{b\Lambda})$ is the direct sum of $\chi_{c\Lambda}$ where $|a-b|\leq c\leq a+b$ and $a+b+c\equiv 0\pmod{2}$. The tuples $(a,b,c)$ with the above property are known as \emph{admissible triples}. More explicitly, the tensor product of $\mathfrak{sl}_2$-modules is given by`
$$\chi_{a\Lambda}\otimes\chi_{b\Lambda}=\chi_{|a-b|\Lambda}\oplus \chi_{(|a-b|+2)\Lambda}\oplus\cdots\oplus \chi_{(a+b)\Lambda}.$$
Assume that $a\geq b$. Since $a+b\leq 2n$ we have two cases: $a+b\leq n$ or $n<a+b\leq 2n$. The above discussion proves the following.

\bigskip

\begin{propo}
	For $\mathfrak{g}$ of type $A_2^{(2)}$ at level $2n$, the product structure of $V_{2n}(\mathfrak{g})$ is explicitly given by
	$$\chi_{a\Lambda}\otimes \chi_{b\Lambda}=\sum_{0\leq i\leq \frac{n+b-a}{2}}\chi_{(a-b+2i)\Lambda}-\sum_{\frac{n+b-a}{2}< i\leq b}\chi_{(2n+1+b-2i-a)\Lambda}.$$
\end{propo}

\bigskip

It follows that $N_{ab}^c<0$ whenever $(a,b,c)$ is admissible, $a+b>n$, and $c>n.$ In particular, $N_{n,n}^{1}=-1$ when $n$ is even and $N_{n,n}^{2}=-1$ when $n$ is odd. Later on, in section \ref{quotient-a2even}, we will give a more explicit description of the product structure of $V_{2n}(A_{2\ell}^{(2)})$ for any rank $\ell$. 

\bigskip

The following curious fact shows that when the level is $2n$ and $2n+1$, the underlying character spaces are equal while the Verlinde algebra structures are quite different. Moreover, I observed that, at even level, if the structure constants $N_{\lambda\mu}^{\nu}$ are replaced by their absolute values $|N_{\lambda\mu}^{\nu}|$, then the resulting algebra is \emph{also} associative and commutative.

\bigskip

\begin{lemma}
	Let $k\geq 1,$ then $\overline{P}^{2n}_+=\overline{P}^{2n+1}_+.$
\end{lemma}

\begin{proof}
	It is a straightforward check to show that $\overline{P}^{2n}_+\subseteq \overline{P}^{2n+1}_+.$ Suppose that $\lambda\in P^{2n+1}_+$ and $\lambda=\sum_{i=0}^{\ell}c_i\Lambda_i,$ then $c_i\in \N$ for $0\leq i\leq \ell$ and
	$$\langle \lambda,K\rangle = \left\langle \lambda,\alpha_0^\vee+2\sum_{i=1}^{\ell}\alpha_i^\vee\right\rangle = 2n+1.$$
	This equation implies that $c_0>0$ and so we can define $\lambda'\in P_{2n}^+$ by $\lambda'=\lambda-\Lambda_0$. Therefore, the earlier inclusion and $\overline{P}^{2n}_+\supseteq \overline{P}^{2n+1}_+$ implies the lemma.
\end{proof}

\bigskip

Another different feature in this case is that the grading structure (as an algebra) only exists when the level is odd (since it exists in $V_k(C_{\ell}^{(1)})$). Nevertheless, when the level is even we can still define the same quotient $\mathcal{A}=\overset{\circ}{P}/\overset{\circ}{Q}\cong \Z_2$ and get the decomposition $V_k(\mathfrak{g})=V_0\bigoplus V_1$ \emph{as a vector space}. Empirically, we have observed the following phenomenon and conjecture that $V_k(\mathfrak{g})$ can be realized as a commutative associative algebra with nonnegative integral structure constants.

\begin{conjecture}
	Let $\mathfrak{g}$ be an affine Lie algebra of type $A_{2\ell}^{(2)}$ and fix the level $2n$ for $n\in \Z_{>0}$. For any $\lambda,\mu\in P_{2n}$, 
	$$\chi_{\lambda}\chi_{\mu}=\sum_{\nu\in P_{2n}^+\text{ mod }\C\delta} (-1)^{[\lambda]+[\mu]+[\nu]} |N_{\lambda\mu}^{\nu}|\chi_{\nu}$$
	in $V_{2n}(\mathfrak{g})$, where $[\phi]$ is the class of $\phi$ in $\mathcal{A}\cong\{0,1\}$. In particular, the structure constants of $V_{2n}(\mathfrak{g})$, with respect to the basis $\{\widetilde{\chi}_{\lambda}\}_{\lambda\in P_{2n}^+\text{ mod }\C\delta}$, where $\widetilde{\chi}_{\lambda}=(-1)^{[\lambda]}\chi_{\lambda}$, are nonnegative.
\end{conjecture}

\bigskip

\section{Applications}
\label{applications section}
\subsection{Quotients of $V_k(\mathfrak{g})$}
\label{quotient-verlinde}
The definition of Verlinde algebras for the various types of affine Lie algebras gives an explicit way of computing the structure constants of these algebras. In a recent paper by J. Hong [Ho2], the author attaches a fusion ring to all affine Lie algebras and he expects that the corresponding structure constants are all non-negative. Here we show that Hong's fusion algebras can be realized as \emph{quotients} of the Verlinde algebras presented here. With this realization, we can answer some questions posed in [Ho2]. In the case of $A_{2\ell}^{(2)}$, the Verlinde algebras and Hong's fusion algebras are isomorphic, proving that the structure constants are sometimes negative (specifically, when the level is even they are negative in all known cases). In the remaining cases, computations show that the structure constants are also sometimes negative (depending on the level and type).

First, let us describe the quotients of the Verlinde algebras to be considered. Recall the map 
$$F_k:\mathcal{R}(\overset{\circ}{\mathfrak{g}})\to Ch_k(\mathfrak{g})$$
given in section \ref{r=1 section} where $\mathfrak{g}$ is an untwisted affine Lie algebra of rank $\ell$. This map depends on the affine Weyl group 
$$\overset{\circ}{W}\ltimes \overset{\circ}{Q}{}^\vee$$
(where, as usual, $\overset{\circ}{Q}{}^\vee$ is identified with a lattice in $\mathfrak{h}^*$ using the invariant bilinear form.) Consider, instead, the group $W':=\overset{\circ}{W}\ltimes \overset{\circ}{Q}$. When $\mathfrak{g}$ has only 1 root length, $W=W'$ and so we may assume that $\mathfrak{g}$ is of type $B_{\ell}^{(1)}, C_{\ell}^{(1)},F_{4}^{(1)},$ or $G_{2}^{(1)}$. The group $W'$ is generated by $s_{\alpha_i'}$ where $\alpha_i'=\alpha_i$ for $i=1.\ldots,\ell$ and $\alpha_0'=\delta-\theta_s,$ where $\theta_s$ is the highest short root of $\overset{\circ}{\mathfrak{g}}$. The $\alpha_i'$ give a root system (that is equivalent to that of $\mathfrak{g}^{t}{}'$) and there are coresponding notions for fundamental weights $\Lambda_0',\Lambda_1',\ldots,$ the element $\rho'=\sum_{i=0}^{\ell}\Lambda_i'$, the weight lattice $P'=\bigoplus_{i=0}^{\ell}\Z \Lambda_i'$, the level $k$ dominant weights $P'^{+}_k$, the space of `primed' level $k$ characters $Ch_k'(\mathfrak{g})=\bigoplus_{\lambda\in P_k'^+\text{ mod }\C\delta}\C A_{\lambda'+\rho'}/A_{\rho'}$, etc. The corresponding dual Coxeter number $h^\vee{}'$ coincides with the Coxeter number $h$. Note that $\overline{\rho}'=\overline{\rho}$. 

The map $F_{k+h-h^\vee}$ uses the action of $\overset{\circ}{W}\ltimes (k+h)\overset{\circ}{Q}{}^\vee$ in its definition. Using the `primed' version of the affine Weyl group $W_{aff}'^{k+h}=\overset{\circ}{W}\ltimes (k+h)\overset{\circ}{Q}{}$, define the map
$$F_k':\mathcal{R}(\overset{\circ}{\mathfrak{g}})\to Ch_k'(\mathfrak{g})$$
$$\overset{\circ}{\chi}_{\overline{\lambda}}\mapsto \left\{\begin{matrix}
\epsilon(w)\chi_{k\Lambda_0'+w(\overline{\lambda}+\overline{\rho})-\overline{\rho}} &\text{ if there exists }w\in W_{aff}'^{k+h},\ w(\overline{\lambda}+\overline{\rho})\in P_{k+h}'^{aff\ ++}\\
0 & \text{ otherwise}
\end{matrix}\right. .$$
Since $\overset{\circ}{Q}{}^\vee \subseteq \overset{\circ}{Q}$, $W_{aff}^{k+h}\subseteq W_{aff}'^{k+h}$ and $F_k'$ factors through $V_{k+h-h^\vee}(\mathfrak{g})$, i.e., there is a map
$$G_k:V_{k+h-h^\vee}(\mathfrak{g})\to Ch_k'(\mathfrak{g})$$
such that $F_k'=G_k\circ F_{k+h-h^\vee}$. The following theorem appears in [Ho2].

\bigskip

\begin{theorem}(Hong)
	\label{hong-homo}
	$F_k'$ is a homomorphism.
\end{theorem}

\bigskip

Recall the correspondence $\mathfrak{g}\leftrightarrow\mathfrak{g}'$ for the twisted affine algebras given in the table in section \ref{r>a_0 case}. Note that the underlying simple Lie algebra of $\mathfrak{g}'$ is isomorphic to the underlying simple Lie algebra of $\mathfrak{g}^t$. Identify their Cartan subalgebras $\overset{\circ}{\mathfrak{h}}{}'=\overset{\circ}{\mathfrak{h}}{}^t$ and notice that since the $\theta$ element of $\mathfrak{g}'$ is the highest short root, $\alpha_0\in \mathfrak{h}'{}^*$ and $(\alpha_0^t)'\in \mathfrak{h}^t{}^*$ can also be identified. This identification induces a natural bijection $P'_k\cong (P_k^t)'$. Finally, note that $h'^\vee=(h^{\vee t})'=h^t$. All of this shows that there is a natural isomorphism $\chi_{\lambda}\mapsto \chi_{\lambda}$ between $Ch_k(\mathfrak{g}')$ and $Ch_k'(\mathfrak{g}^t)$. Using this natural isomorphism, we have the following result.

\bigskip

\begin{theorem}
	\label{quotient-verlinde}
	$G_k$ is a surjective homomorphism. In particular, if $\mathfrak{g}$ is a twisted affine Lie algebra that is not of type $A_{2\ell}^{(2)}$, then $Ch_k(\mathfrak{g}')=G_k(V_k(\mathfrak{g}))$ has the structure of a commutative associative algebra.
\end{theorem}

\begin{proof}
	Let $\mathfrak{g}$ be an untwisted affine algebra and $k$ a positive integer. Since $F_{k+h-h^\vee}$ is surjective, then so is $G_k$. For any $\chi_{\lambda},\chi_{\mu}\in V_{k+h-h^\vee}(\mathfrak{g})$, recall that $\chi_{\lambda}\chi_{\mu}=F_{k+h-h^\vee}(\overset{\circ}{\chi}_{\overline{\lambda}}\overset{\circ}{\chi}_{\overline{\mu}})$. Therefore, by 	\ref{hong-homo},
	$$G_k(\chi_{\lambda}\chi_{\mu})=G_k(F_{k+h-h^\vee}(\overset{\circ}{\chi}_{\overline{\lambda}}\overset{\circ}{\chi}_{\overline{\mu}}))=F_k'(\overset{\circ}{\chi}_{\overline{\lambda}})F_k'(\overset{\circ}{\chi}_{\overline{\mu}})=G_k(\chi_{\lambda})G_k(\chi_{\mu}).$$
	In particular, since $V_{k}(\mathfrak{g})=V_{k+h^\vee-h}(\mathfrak{g}^t)$ (notice the crucial shift of level) and $h^\vee=h^t, h=h^{\vee t}$, the discussion in the preceding paragraph proves that the quotient algebra $Ch'_k(\mathfrak{g}^t)$ is naturally isomorphic to $Ch_k(\mathfrak{g}')$ and so the latter can be given the same algebra structure.
\end{proof}

Let $\mathfrak{g}$ be a twisted affine Kac-Moody Lie algebra that is not of type $A_{2\ell}^{(2)}$ and let $k$ a positive integer, then Hong's fusion algebra $R_k(\mathfrak{g})$ is defined to be the vector space 
$$R_k(\mathfrak{g})=\bigoplus_{\lambda\in P_k\text{ mod }\C\delta}\C \overset{\circ}{\chi}_{\overline{\lambda}}$$
with structure constants
$$c_{\lambda\mu}^{\nu}=\sum_{w\in W_{k}^\dagger}\epsilon(w) N_{\overline{\lambda},\overline{\mu}}^{\overline{\nu}}$$
where $W_{k}^\dagger$ is the set of minimal representations of the left cosets of $\overset{\circ}{W}$ in $W_{aff}^{k+h^\vee}$ and $N_{\overline{\lambda},\overline{\mu}}^{\overline{\nu}}$ is the coefficient of $\overset{\circ}{\chi}_{\overline{\nu}}$ in $\overset{\circ}{\chi}_{\overline{\lambda}}\overset{\circ}{\chi}_{\overline{\mu}}$. When $\mathfrak{g}$ is of type $A_{2\ell}^{(2)}$, we use the same definition of $R_k(\mathfrak{g})$ except that we use the affine Weyl group $\overset{\circ}{W}\ltimes (k+h^\vee)\overset{\circ}{Q}{}^\vee$. See [Ho2] for more details. Now, we prove that $R_k(\mathfrak{g}')$ is a quotient of $V_k(\mathfrak{g})$ and give an example of how to compute the structure constants $c_{\lambda\mu}^{\nu}$ from the fusion rules $N_{\lambda\mu}^{\nu}$. 

\bigskip

\begin{propo}
	Let $\mathfrak{g}$ be a twisted affine Lie algebra that is not of type $A_{2\ell}^{(2)}$. The fusion algebra $R_k(\mathfrak{g}')$ is isomorphic to a quotient of $V_{k}(\mathfrak{g}).$ More precisely, $R_k(\mathfrak{g}')\cong G_k(V_{k}(\mathfrak{g}))$. When $\mathfrak{g}$ is of type $A_{2\ell}^{(2)}$, then $R_k(\mathfrak{g})\cong V_k(\mathfrak{g})$.
\end{propo}

\begin{proof}
	First suppose that $\mathfrak{g}$ is a twisted affine algebra but is not of type $A_{2\ell}^{(2)}$. As vector spaces, $Ch_k(\mathfrak{g}')\cong R_k(\mathfrak{g}')$. Since for any $\lambda \in P_k',$ $w(\overline{\lambda}+\overline{\rho})-\overline{\rho}\in \overset{\circ}{P}{}'{}^+$ if and only if $w\in W_k^\dagger$, the coefficient of $\chi_{\nu}\in Ch_k(\mathfrak{g}')$ in $F_k'(\overset{\circ}{\chi}_{\overline{\lambda}}\overset{\circ}{\chi}_{\overline{\mu}})$ is $c_{\lambda\mu}^{\nu}$. The result then follows from \ref{quotient-verlinde}.\\
	\\
	When $\mathfrak{g}$ is of type $A_{2\ell}^{(2)}$, the underlying vector spaces are similarly isomorphic and the homomorphism $F_k$ gives the exact same structure constants $c_{\lambda\mu}^{\nu}$.
\end{proof}

\bigskip

\begin{corollary}
	Let $\mathfrak{g}$ be a twisted affine Lie algebra, then the structure constants $c_{\lambda\mu}^{\nu}$ of $R_k(\mathfrak{g}')$ are given by
	$$c_{\lambda\mu}^{\nu}=\sum_{w\in X_{\nu}}\sum_{\phi\in P'^{k+}\text{ mod }\C\delta}\frac{S_{\lambda\phi}S_{\mu\phi}(S^{-1})_{\nu\phi}}{S_{k\Lambda_0,\phi}},$$
	where $X_\nu=\{w\in W_k^\dagger:(w(\overline{\nu}+\overline{\rho})-\overline{\rho},\theta^t)\leq k+h^\vee-h\}$.
\end{corollary}

\bigskip

\begin{example}
	The left multiplication matrices of $G_k(V_k(\mathfrak{g}))$ are given by `truncating and folding' the left multiplication matrices of $V_k(\mathfrak{g})$. More precisely, one removes all columns corresponding to weights $\lambda$ such that $\langle \lambda+\rho,\theta^\vee\rangle\geq k+h^\vee$ (here $\theta^\vee$ is the highest coroot of $\overset{\circ}{\mathfrak{g}}$), removes all rows corresponding to $\lambda$ such that $\langle \lambda+\rho,\theta^\vee\rangle$ is an integral multiple of  $k+h^\vee$, and then adds all rows that are in the same orbit under the \emph{dot action} $w\cdot \lambda=w(\lambda+\rho)-\rho$ with appropriate sign $\epsilon(w)$. In particular, consider the data from example \ref{B3-example} where the weights are
	$$0,\overline{\Lambda}_1,2\overline{\Lambda}_1,\overline{\Lambda}_2,2\overline{\Lambda}_3,$$
	$$\overline{\Lambda}_3,\overline{\Lambda}_1+\overline{\Lambda}_3.$$
	In this case, $\mathfrak{g}$ is of type $A_{5}^{(2)}$ and $k=1$.
	The only $\overline{\lambda}+\overline{\rho}$, for $\overline{\lambda}$ from the list above, satisfying $\langle \lambda+\rho,\theta^\vee\rangle< k+h^\vee=1+6$ are $0$ and $\overline{\Lambda}_3$. The weights $\overline{\Lambda}_1+\overline{\rho}$, $\overline{\Lambda}_2+\overline{\rho},$ and $2\overline{\Lambda}_3$ satisfy $\langle \lambda+\rho,\theta^\vee\rangle=7$ and so are thrown out. Finally, adding the remaining rows (all zeros) we get
	$$L_1'=\begin{pmatrix}
		1	&	0\\
		0	&	1
	\end{pmatrix} \text{ and }L_2'=\begin{pmatrix}
	0	&	1\\
	1	&	0
	\end{pmatrix}.$$
\end{example}
 
\bigskip

\subsubsection{More Negative Structure Constants}

In the remaining twisted cases, negative structure constants also appear in a pattern very similar to the appearance of negative structure constants for the Verlinde algebra of $A^{(2)}_{2\ell}.$ Unlike the conjectured state of affairs for $A^{(2)}_{2\ell},$ experimental data shows that it is not possible to `twist' the basis by a character in order to make the structure constants non-negative. Nonetheless, a grading structure plays a role in determining when the structure constants are non-negative. 

Let $\mathfrak{g}$ be a twisted affine Lie algebra that is not of type $A^{(2)}_{2\ell}.$ The algebra $R(\mathfrak{g})$ defined in the previous section does not inherit the algebra grading of $V_k(\mathfrak{g}').$ Instead, there is a \emph{vector space grading} indexed by $\overset{\circ}{P}/\overset{\circ}{Q}$ analogous to the untwisted case, i.e., 
$$R(\mathfrak{g})=\bigoplus_{a\in \mathcal{A}}R(\mathfrak{g})_a$$
where $R(\mathfrak{g})_a=\text{span}\{\chi_{\lambda}:\lambda\in P_k^+\text{ and } \overline{\lambda}\in a\}.$ Similar to the $A^{(2)}_{2\ell}$ case, this vector space grading seems to play a role (albeit different) in determining the non-negativity of the structure constants of $R(\mathfrak{g}).$ The structure constants $N_{\lambda\mu}^{\nu}$ of $R(\mathfrak{g})$ can be observed to behave according to the following $2/3$'s rule.

\bigskip

\begin{conjecture}
	Let $k$ be an even positive integer and $\mathfrak{g}$ be a twisted affine Lie algebra that is not of type $A^{(2)}_{2\ell}.$ If $\lambda,\mu,\nu\in P_k^+$ and at least 2 out of   $\overline{\lambda},\overline{\mu},\overline{\nu}$ lie in $\overset{\circ}{Q},$ then $N_{\lambda\mu}^{\nu}\geq 0,$ where $N_{\lambda\mu}^{\nu}$ are the structure constants of $R(\mathfrak{g})$ with respect to the basis $\{\chi_\lambda\}_{\lambda\in P_k^+}.$
\end{conjecture}

\bigskip

\subsection{The $A^{(2)}_{2\ell}$ Quotient Structure}
\label{quotient-a2even}

In the case where $\mathfrak{g}$ is of type $A^{(2)}_{2\ell}$, the Verlinde algebra $V_{k}(\mathfrak{g})$ can be realized as a quotient of a Verlinde algebra $V_{k'}(C^{(1)}_{\ell})$. In this section, we use the affine Weyl groups in question to explicitly describe the quotient $V_k(\mathfrak{g})$. We then explain the relationship between the grading structures of these vector spaces.  

\bigskip

\begin{lemma}
	Let $\mathfrak{g}$ be of type $A^{(2)}_{2\ell}$ and fix a positive integer $k$. Let $\widetilde{\mathfrak{g}}$ be an affine algebra of type $C^{(1)}_{\ell}$, then there exist a surjective homomorphism
	$$F:V_{k+\ell}(\widetilde{\mathfrak{g}})\to V_k(\mathfrak{g})$$
	given by $F(\chi_{\lambda})=F_k\overset{\circ}{\chi}_{\overline{\lambda}}$, where $F_k$ is the homomorphism defined at the end of section \ref{r=1 section}.
\end{lemma}

\begin{proof}
	The underlying simple Lie algebras of $\mathfrak{g}$ and $\widetilde{\mathfrak{g}}$ coincide and so we will identify their Cartan subalgebras and their duals. The map $F$ is clearly well-defined as a linear map. To show that it is a homomorphism, it is enough to show that the homomorphism $F_k$ factors through $V_{k+\ell}(\widetilde{\mathfrak{g}})$. Indeed, $V_{k+\ell}(\widetilde{\mathfrak{g}})$ is the quotient of the character ring $\mathcal{R}(\overset{\circ}{\mathfrak{g}})$ by the dot action of $\overset{\circ}{W}\ltimes (k+2\ell+1)\widetilde{\nu}(\overset{\circ}{Q}{}^\vee)$ where $\widetilde{\nu}$ is the map $\mathfrak{h}\to \mathfrak{h}^*$ induced by the normalized bilinear form of $\widetilde{\mathfrak{g}}$. If $\nu$ is the normalized bilinear form of $\mathfrak{g}$, then note that 
	$$\widetilde{\nu}(\overset{\circ}{Q}{}^\vee)=2\nu(\overset{\circ}{Q}{}^\vee)\subset \overset{\circ}{Q}\subset \nu(\overset{\circ}{Q}{}^\vee).$$
	Since $V_{k}(\mathfrak{g})$ is the quotient of $\mathcal{R}(\overset{\circ}{\mathfrak{g}})$ by the dot action of $\overset{\circ}{W}\ltimes (k+2\ell+1)\nu(\overset{\circ}{Q}{}^\vee)$, it follows that $F_k$ factors through $V_{k+\ell}(\widetilde{\mathfrak{g}})$. 
\end{proof}

\bigskip

\begin{remark}
	Let $A_1$ be the fundamental alcove of the affine Weyl group. Note that the fundamental domain of $\overset{\circ}{W}\ltimes (k+2\ell+1)\nu(\overset{\circ}{Q}{}^\vee)$ is $(k+2\ell+1)A_{1}$ and the fundamental domain of $\overset{\circ}{W}\ltimes (k+2\ell+1)\widetilde{\nu}(\overset{\circ}{Q}{}^\vee)$ is $2(k+2\ell+1)A_{1}$. We will soon present an algorithm mapping elements from $2(k+2\ell+1)A_{1}$ into $(k+2\ell+1)A_{1}$. 
\end{remark}

\bigskip

Let $\mathfrak{g}$ be of type $A^{(2)}_{2\ell}$ as before. For the next result, let us use the notation $w_0:=id$ and
$$w_i:=s_{i-1}s_{i-2}\cdots s_0$$
for $0< i\leq \ell$ where $s_i\in W$ are the Coxeter generators of the Weyl group $W$ of $\mathfrak{g}$. For any list of integers $I=(i_1,\ldots,i_m)$, also define
$$w_I=w_{i_m}w_{i_{m-1}}\cdots w_{i_1}.$$
In the proof of this next theorem, we define an explicit recursive algorithm defining the quotient indicated by the previous lemma.

\begin{theorem}
	Let $\mathfrak{g}$ be of type $A^{(2)}_{2\ell}$, fix a positive integer $k$, and let $\widetilde{\mathfrak{g}}$ be an affine algebra of type $C^{(1)}_{\ell}$. If $\lambda$ is a dominant weight of $\widetilde{\mathfrak{g}}$ of level $k+\ell$ such that $\overline{\lambda}+\overline{\rho}$ has trivial stabilizer with respect to the affine Weyl group $\overset{\circ}{W}\ltimes \nu(\overset{\circ}{Q}{}^\vee)$, then there is a weight $\mu$ of $\mathfrak{g}$ of level $k$ and a sequence of nonnegative integers $I=(i_1,\ldots,i_m)$, where $m$ satisfies $1\leq m\leq k+2\ell+1$, such that
	$$\overline{\mu}+\overline{\rho}=w_I(\overline{\lambda}+\overline{\rho}).$$
	In particular, for the homomorphism given in the previous lemma and $\lambda$ as above,
	$$F(\chi_{\lambda})=(-1)^{i_1+\cdots + i_m} F_k(\overset{\circ}{\chi}_{w_I\cdot \overline{\lambda}}).$$
\end{theorem}

\begin{proof}
	Let $k$ be as in the statement and recall that $h^\vee=2\ell+1$. Also recall that $\theta=2\alpha_1+\cdots +2\alpha_{\ell-1}+\alpha_{\ell}$ and $\theta^\vee=\alpha_1^\vee+\cdots +\alpha_{\ell}^\vee$. We identify $W\cong \overset{\circ}{W}\ltimes (k+h^\vee)\nu(\overset{\circ}{Q}{}^\vee)$. The Weyl group $W$ of $\mathfrak{g}$ contains the subgroup $\overset{\circ}{W}\ltimes 2(k+h^\vee)\nu(\overset{\circ}{Q}{}^\vee)$ and maps every dominant weight $\overline{\lambda}\in \overset{\circ}{P}{}^+$ to a unique dominant weight $\overline{\mu}$ such that $\langle \overline{\mu},\theta^\vee\rangle =\frac{1}{2}(\overline{\mu},\theta)\leq k+h^\vee$. Therefore, for the purpose of describing $F$, we only need to consider dominant weights $\overline{\mu}$ such that $(\overline{\mu}+\overline{\rho},\theta)< 2(k+h^\vee)$.
		
	Let $\lambda\in P^+_{k+\ell}(\widetilde{\mathfrak{g}})$ be a \emph{regular} dominant weight, i.e., a weight such that $\overline{\lambda}+\overline{\rho}$ has trivial stabilizer with respect to $\overset{\circ}{W}\ltimes (k+h^\vee)\nu(\overset{\circ}{Q}{}^\vee)$, then there is a unique element $w_\lambda\in \overset{\circ}{W}\ltimes (k+h^\vee)\nu(\overset{\circ}{Q}{}^\vee)$ such that $w_{\lambda}(\overline{\lambda}+\overline{\rho})\in \overset{\circ}{P}{}^{+}$, and $(w_{\lambda}(\overline{\lambda}+\overline{\rho}),\theta)< k+h^\vee$. If $(\overline{\lambda}+\overline{\rho},\theta)< k+h^\vee$, then we can take $w_{\lambda}=w_0=id$.
	
	Suppose that $k+h^\vee< (\overline{\lambda}+\overline{\rho},\theta)<2(k+h^\vee)$. Write 
	$$\overline{\lambda}+\overline{\rho}=\sum_{i=1}^{\ell}b_i\alpha_i^\vee=\sum_{i=1}^{\ell}c_i\overline{\Lambda}_i$$
	(where $\alpha_i^\vee$ is mapped into $\overset{\circ}{\mathfrak{h}}{}^*$ via $\nu$) such that $b_i>0, c_i>0$ for $i=1,\ldots,\ell$, then $k+h^\vee< 2b_1<2(k+h^\vee)$ since $\theta=2\overline{\Lambda}_1$. It is straightforward to check that $b_1=c_1+\cdots + c_{\ell}$. Applying the affine generator $s_0$, we get
	\begin{align*}
	s_0(\overline{\lambda}+\overline{\rho})&=s_{\theta}(\overline{\lambda}+\overline{\rho})+(k+h^\vee)\theta^\vee\\
	&=\overline{\lambda}+\overline{\rho}-\langle\overline{\lambda}+\overline{\rho},\theta^\vee\rangle \theta+(k+h^\vee)\theta^\vee\\
	&=\overline{\lambda}+\overline{\rho}+(k+h^\vee-2b_1)\theta^\vee\\
	&=(k+h^\vee-2b_1+c_1)\overline{\Lambda}_{1}+\sum_{i=2}^{\ell}c_i\overline{\Lambda}_{i}.
	\end{align*}
	It follows that $(s_0(\overline{\lambda}+\overline{\rho}),\theta)=2(k+h^\vee-b_1)<k+h^\vee$ and, therefore, if $k+h^\vee-2b_1+c_1>0$, then $w_{\lambda}=w_{1}=s_0$. 
	
	Suppose that $k+h^\vee-2b_1+c_1<0$.	Compute the following weights
	\begin{align*}
	s_1s_0(\overline{\lambda}+\overline{\rho})=&-(k+h^\vee-2b_1+c_1)\overline{\Lambda}_{1}+(k+h^\vee-2b_1+c_1+c_2)\overline{\Lambda}_2+\cdots+c_{\ell}\overline{\Lambda}_\ell\\
	s_2s_1s_0(\overline{\lambda}+\overline{\rho})=&c_2\overline{\Lambda}_{1}-(k+h^\vee-2b_1+c_1+c_2)\overline{\Lambda}_2+(k+h^\vee-2b_1+c_1+c_2+c_3)\overline{\Lambda}_3+\cdots\\
	s_3s_2s_1s_0(\overline{\lambda}+\overline{\rho})=&c_2\overline{\Lambda}_{1}+c_3\overline{\Lambda}_2-(k+h^\vee-2b_1+c_1+c_2+c_3)\overline{\Lambda}_3+\cdots\\
	s_{\ell-1}\cdots s_1s_0(\overline{\lambda}+\overline{\rho})=&c_2\overline{\Lambda}_{1}+c_3\overline{\Lambda}_2+\cdots+c_{\ell-1}\overline{\Lambda}_{\ell-2}\\
	&-(k+h^\vee-2b_1+c_1+\cdots +c_{\ell-1})\overline{\Lambda}_{\ell-1}
	+(k+h^\vee-b_1)\overline{\Lambda}_{\ell}.
	\end{align*}
	If we set $\overline{\lambda}_i=s_{i-1}s_{i-2}\cdots s_0(\overline{\lambda}+\overline{\rho})=w_i(\overline{\lambda}+\overline{\rho})$, then it must be the case that $\overline{\lambda}_{i_1}\in \overset{\circ}{P}{}^{+}$ for some $i_1$ such that $0\leq i_1\leq \ell$. In fact, $\overline{\lambda}_j$ is a regular dominant weight as soon as $2b_1-\sum_{i=1}^{j}c_i\leq k+h^\vee$ and, by assumption, $b_1=2b_1-\sum_{i=1}^{\ell}c_i<k+h^\vee$. Moreover, for $i>1$, $(\overline{\lambda}_i,\theta)=2(b_1-c_1)<2b_1$. If the element $w_{i_1}(\overline{\lambda}+\overline{\rho})$ satisfies $(w_{i_1}(\overline{\lambda}+\overline{\rho}),\theta)<k+h^\vee$ then take $w_{\lambda}=w_{i_1}$ and we are done. Otherwise, we may continue in this fashion, producing a sequence
	$$(\overline{\lambda}+\overline{\rho},\theta)>(w_{i_1}(\overline{\lambda}+\overline{\rho}),\theta)>(w_{i_2}w_{i_1}(\overline{\lambda}+\overline{\rho}),\theta)>\cdots$$
	which must eventually lead to $(w_I(\overline{\lambda}+\overline{\rho}),\theta)<k+h^\vee$ in at most $k+h^\vee$ steps for some finite sequence $I=(i_1,\ldots,i_m)$.
\end{proof}

\bigskip

\begin{corollary}
	Let $\mathfrak{g}$, $k$, $\widetilde{\mathfrak{g}}$, $\lambda$, $I$, $m$ etc. be as in the previous theorem. Recall that $V=V_k(\mathfrak{g})$ has a $\Z_2\cong \overset{\circ}{P}/\overset{\circ}{Q}$ grading as a vector space. If $k$ is even, then 
	$$F(\chi_{\lambda})\in V_{m+[\lambda] \text{ mod }2}$$
	if all $i_j>0$ for $1\leq j\leq m$, where $[\lambda]\in\Z_2$ is the class of $\lambda$ in $\overset{\circ}{P}/\overset{\circ}{Q}$.
\end{corollary}

\begin{proof}
	The lattice $\nu(\overset{\circ}{Q}{}^\vee)$ is spanned by $\alpha_1,\ldots,\alpha_{\ell-1},\frac{1}{2}\alpha_{\ell}$, $2\nu(\overset{\circ}{Q}{}^\vee)\subseteq \overset{\circ}{Q}$, and 
	$$\overset{\circ}{W}\left(\frac{1}{2}\alpha_{\ell}+\overset{\circ}{Q}\right)=\frac{1}{2}\alpha_{\ell}+\overset{\circ}{Q}.$$
	We can express 
	$$w_i=s_{i-1}\cdots s_0 = t_{s_{i-1}\cdots s_1\theta^\vee} s_{i-1}\cdots s_1 s_\theta$$
	$$=t_{\alpha_i^\vee+\cdots+\alpha_{\ell}^\vee } s_{i-1}\cdots s_\theta$$
	for any $i>0,$ and so $w_i\overline{\lambda}$ is in class $[\lambda]+1$. The result follows immediately.
\end{proof}

\bigskip

\subsection{Congruence Group Action}

As a another application, let $\mathfrak{g}$ be a twisted affine Lie algebra of type $X_N^{(r)}$ and $k$ a positive integer, then $Ch_k(\mathfrak{g})$ is a module for the congruence subgroup $\Gamma_1(r)$ (see \ref{theta-fns-characters}) [Ka].

\bigskip

\begin{propo}
	\label{u21-action-prop}
	Let $\mathfrak{g}$ be as above and $\lambda\in P^k$, then
	$$\chi_{\lambda}|_{u_{21}^r}=v_k\sum_{\mu\in P_k\text{ mod }\C\delta,\nu\in P^k\text{ mod }\C\delta}e^{-2r\pi i m_{\overline{\mu}}}(S^t_{0\mu})^2\overset{\circ}{\chi}_{\overline{\lambda}}{}^t\left(e^{-\frac{2\pi i (\overline{\mu}+\overline{\rho})}{k+h^\vee}}\right)\overset{\circ}{\chi}_{\overline{\nu}}{}^t\left(e^{\frac{2\pi i (\overline{\mu}+\overline{\rho})}{k+h^\vee}}\right)\chi_{\nu}$$
	where $v_k$ is a nonzero scalar depending only on $k$ and $\mathfrak{g}$. In particular, when $r=1$
	$$\chi_{\lambda}|_{u_{21}}=\sum_{\mu\in P_k\text{ mod }\C\delta}e^{2\pi i (m_{\overline{\lambda}}+m_{\overline{\mu}})}S_{k\Lambda_0,,\mu}\overset{\circ}{\chi}_{\overline{\lambda}}\left(e^{-\frac{2\pi i (\overline{\mu}+\overline{\rho})}{k+h^\vee}}\right)\chi_{\mu}.$$
\end{propo}

\begin{proof}
	These follow directly from \ref{gamma1r-morphism} and the identities $u_{21}^r=Su_{12}^rS^{-1}$ and $u_{21}=u_{12}^{-1}Su_{12}^{-1}$, respectively.
\end{proof}

\bigskip

\begin{corollary}
	\label{u21-identity}
	Let $\overset{\circ}{\mathfrak{g}}$ be a finite dimensional simple Lie algebra with weight lattice $\overset{\circ}{P}$ and $\overset{\circ}{\chi}_{\lambda}$ be the character of the irreducible finite dimensional module of highest weight $\lambda\in \overset{\circ}{P}{}^+,$ then, for any $\lambda,\mu\in \overset{\circ}{P}{}^+,$
	$$\sum_{\nu\in \overset{\circ}{P}}e^{-2\pi i m_{\nu}}(S_{k\Lambda_0,\nu})^2\overset{\circ}{\chi}_{\lambda}\left(e^{-\frac{2\pi i (\nu+\rho)}{k+h^\vee}}\right)\overset{\circ}{\chi}_{\mu}\left(e^{\frac{2\pi i (\nu+\rho)}{k+h^\vee}}\right)=e^{2\pi i (m_{\lambda}+m_{\mu})}S_{k\Lambda_0,\mu}\overset{\circ}{\chi}_{\lambda}\left(e^{-\frac{2\pi i (\mu+\rho)}{k+h^\vee}}\right),$$
	where $m_{\phi}$ is defined the same way as the modular anomaly is for affine algebras and $S_{k\Lambda_0,\phi}=|M^*/(k+h^\vee)M|^{-1/2}\prod_{\alpha\in \overset{\circ}{\Delta}_+}2\sin\frac{\pi(\phi+\rho,\alpha)}{k+h^\vee}$.
\end{corollary}

\bigskip

The evaluations the characters of $\overset{\circ}{\mathfrak{g}}{}^t$ in \ref{u21-action-prop} can be replaced by multiples of the characters of $\overset{\circ}{\mathfrak{g}}$. It is worth that noting that the action of $S$ on untwisted Verlinde algebras gives the strikingly similar identity
$$\sum_{\nu\in \overset{\circ}{P}}(S_{k\Lambda_0,\nu})^2\overset{\circ}{\chi}_{\lambda}\left(e^{-\frac{2\pi i (\nu+\rho)}{k+h^\vee}}\right)\overset{\circ}{\chi}_{\mu}\left(e^{\frac{2\pi i (\nu+\rho)}{k+h^\vee}}\right)=\delta_{\lambda\mu}.$$
Indeed, the above transformation property is used to define a nondegenerate Hermitian form on the fusion algebras in other works such as [B], [Ho2]. Pursuing this idea further, one can say more about the above identity in \ref{u21-action-prop}. Let us adopt the notation
$$\left\langle\overset{\circ}{\chi}_{\lambda},\overset{\circ}{\chi}_{\mu}\right\rangle_r=\sum_{\nu\in P_k^{+}\text{ mod }\C\delta}e^{-2r\pi i m_{\overline{\nu}}}(S^t_{0\nu})^2\overset{\circ}{\chi}_{\overline{\lambda}}{}^t\left(e^{-\frac{2\pi i (\overline{\nu}+\overline{\rho})}{k+h^\vee}}\right)\overset{\circ}{\chi}_{\overline{\mu}}{}^t\left(e^{\frac{2\pi i (\overline{\nu}+\overline{\rho})}{k+h^\vee}}\right).$$

\begin{theorem}
	\label{u21-partial-character-identity}
	Let $\mathfrak{g}$ be of type $X_N^{(r)}$, $k$ be a positive integer, and $\lambda,\mu\in P^{k+}$, then there exists a $\beta=\beta(\mathfrak{g},k)\in \overset{\circ}{\mathfrak{h}}{}^*$ and a nonzero scalar $v=v(\mathfrak{g},k)$ such that
	$$\chi_{\lambda}|_{u_{21}^r}=ve^{-\frac{\pi i |\beta|^2}{r(k+h^\vee)}}\sum_{(\mu,w)\in P(\mathfrak{g},k,\lambda)\text{ mod }\C\delta}{\varepsilon(w)e^{\frac{\pi i}{r(k+h^\vee)}|\overline{\mu}+\overline{\rho}-\overline{w(\lambda+\rho)}|^2}\chi_{\mu}}$$
	where the sum is over 
	$$P(\mathfrak{g},k,\lambda)=\{(\mu,w)\in P^{k+}\times \overset{\circ}{W}|\ \overline{\mu}+\overline{\rho}-\beta=w(\overline{\lambda}+\alpha) (\text{mod }(k+h^\vee)M)  \text{ where }\alpha\in rM^*\}.$$
	In particular,
	$$\left\langle\overset{\circ}{\chi}_{\lambda},\overset{\circ}{\chi}_{\mu}\right\rangle_r=\left\{\begin{matrix}
	ve^{\frac{\pi i}{r(k+h^\vee)}(|\overline{\mu}+\overline{\rho}|^2+|\overline{\lambda}+\overline{\rho}|^2-|\beta|^2)}\sum_{w}\varepsilon(w)e^{-\frac{2\pi i}{r(k+h^\vee)}(\overline{\mu}+\overline{\rho},w(\overline{\lambda}+\overline{\rho}))}, & \text{ if }\mu+\rho-\beta\in P^{k++}\\
	0, & \text{otherwise}
	\end{matrix}\right.$$
	where the sum is over $w\in \overset{\circ}{W}$ such that $\overline{\mu}+\overline{\rho}-w(\overline{\lambda}+\overline{\rho})-\beta\in rM^*=r\overline{P}_k$. 
\end{theorem}

\begin{proof}
	By the transformation rules given in [KP,\textsection 3], level $k$ theta functions satisfy
	$$\Theta_{\lambda}|_{u_{21}^r}=v\sum_{\alpha\in M^*,\ r\alpha\pmod{kM}}{e^{\pi i k^{-1}(\alpha,r\alpha+2\beta)}\Theta_{\lambda+r\alpha+\beta}}$$
	for some scalar $v\in \C$ such that $|v|=|((k+h^\vee)M+rM^*)/(k+h^\vee)M|^{-1/2}$ and $\beta\in \overset{\circ}{\mathfrak{h}}{}^*$ such that
	$$kr|\alpha|^2\equiv 2(\alpha,\beta)\ (\text{mod }2\Z)$$
	for all $\alpha\in \frac{1}{r}M\cap \frac{1}{k}M^*$. Applying this to $A_{\lambda+\rho}=\sum_{w\in \overset{\circ}{W}}\epsilon(w)\Theta_{w(\lambda+\rho)}$ and finding the coefficient of $\Theta_{\mu+\rho}$ in the expansion of $A_{\lambda+\rho}|_{u_{21}^r}$ for $\mu\in P^{k+}$ gives the first result. Note that for $(\mu,w)\in P(\mathfrak{g},k,\lambda)$, $w(\alpha) = \overline{\mu}+\overline{\rho}-\beta-w(\overline{\lambda}+\overline{\rho})\in rM^*$ and so 
	$$(w(\alpha),w(\alpha)+2\beta)=|\overline{\mu}+\overline{\rho}-w(\overline{\lambda}+\overline{\rho})|^2-|\beta|^2.$$ 
	The second identity in the theorem statement follows immediately from the first and \ref{u21-action-prop}.
\end{proof}

\bigskip

Notice that when $r=1$, one has that $\beta=0$ and $\left\langle\overset{\circ}{\chi}_{\lambda},\overset{\circ}{\chi}_{\mu}\right\rangle_1$ is a scalar multiple of $\overset{\circ}{\chi}_{\overline{\lambda}}$, matching proposition \ref{u21-identity}. Indeed, $\left\langle\overset{\circ}{\chi}_{\lambda},\overset{\circ}{\chi}_{\mu}\right\rangle_r$ is in general a multiple of a \emph{summand} of the character $\overset{\circ}{\chi}_{\overline{\lambda}}$ evaluated at $\frac{2\pi i(\overline{\mu}+\overline{\rho})}{r(k+h^\vee)}$. Specifically, the summand is
$$\sum_{\substack{w\in\overset{\circ}{W}:\\ \overline{\mu}+\overline{\rho}-w(\overline{\lambda}+\overline{\rho})-\beta\in rM^*}}\epsilon(w) e^{w(\overline{\lambda}+\overline{\rho})}.$$

\section{Appendix}
\label{affine-algebra-prelims}
\label{notations-conventions}
\subsection{Preliminaries}
This section very briefly reminds the reader about the theory of affine Kac-Moody Lie algebras. For greater detail, the reader is encouraged to consult [Ka]. For the sake of simplicity, affine Kac-Moody Lie algebras will often be referred to as affine Lie algebras or affine algebras.

Let $A=(a_{ij})_{0\leq i,j\leq \ell}$ be an irreducible affine (generalized) Cartan matrix and $(\mathfrak{h},\Pi,\Pi^\vee)$ be a realization of $A,$ with
$$\Pi=\{\alpha_0,\alpha_1,\ldots,\alpha_\ell\}\subset \mathfrak{h}^* \text{ and } \Pi^\vee=\{\alpha_0^\vee,\alpha_1^\vee,\ldots,\alpha_\ell^\vee\} \subset \mathfrak{h}.$$ 
Let $(\overset{\circ}{\mathfrak{h}},\overset{\circ}{\Pi},\overset{\circ}{\Pi}{}^\vee)$ be a realization of the associated finite root system whose Cartan matrix is $\overset{\circ}{A}=(a_{ij})_{1\leq i,j\leq \ell}.$ In particular,  
$$\overset{\circ}{\Pi}=\{\alpha_1,\ldots,\alpha_\ell\}\subset \overset{\circ}{\mathfrak{h}}{}^* \text{ and }\overset{\circ}{\Pi}{}^\vee=\{\alpha_1^\vee,\ldots,\alpha_\ell^\vee\}\subset \overset{\circ}{\mathfrak{h}}.$$ The vector spaces $\mathfrak{h}$ and $\mathfrak{h}^*$ have bases $\Pi^\vee \cup \{d\}$ and $\Pi \cup \{\Lambda_0\}$, resp., where $\langle\Lambda_0,\alpha_i^\vee\rangle = \delta_{0i}$ and $\langle \Lambda_0,d\rangle =0$.
The affine Kac-Moody Lie algebra $\mathfrak{g}=\mathfrak{g}(A)$ is generated by $\mathfrak{h}$ and $e_i,f_i$, $0\leq i\leq \ell,$ subject to the Chevalley-Serre relations. The Lie algebra $\overset{\circ}{\mathfrak{g}}$ is the Lie subalgebra of $\mathfrak{g}$ generated by $\overset{\circ}{\mathfrak{h}}$ and the generators $e_i,f_i$, $1\leq i\leq \ell$, along with their corresponding Chevalley-Serre relations. The algebra $\overset{\circ}{\mathfrak{g}}$ is a finite dimensional simple Lie algebra and is referred to as the \emph{underlying simple Lie algebra} of $\mathfrak{g}$. The transpose $A^t$ has the realization $(\mathfrak{h},\Pi^\vee,\Pi)$ and has associated affine Kac-Moody algebra $\mathfrak{g}^t:=\mathfrak{g}(A^t)$ called the \emph{transpose algebra} of $\mathfrak{g}(A).$ 

The irreducible affine Kac-Moody algebras are classified (up to isomorphism) by the affine Dynkin diagrams given in [Ka, Ch. 4]. More precisely, there are three tables of affine algebras: $\text{Aff } 1, \text{Aff } 2,$ and $\text{Aff } 3$. 

If an affine algebra is isomorphic to one from table $\text{Aff } r$, it is said to be of \emph{type} $X_N^{(r)}$ for $X\in\{A,B,C,D,E,F,G\}$ and $N\in \Z_{>0}$. The affine algebras in $\text{Aff } 1$ are called \emph{untwisted affine Lie algebras} while those in $\text{Aff } 2$ and $\text{Aff } 3$ are called \emph{twisted affine Lie algebras}. For any affine Lie algebra $\mathfrak{g}$, the Dynkin diagram of its underlying simple Lie algebra $\overset{\circ}{\mathfrak{g}}$ can be obtained from the affine Dynkin diagram of $\mathfrak{g}$ by removing the $0$ node (and all edges attached to the 0 node).

For an affine irreducible Cartan matrix $A,$ let $D(A)$ denote its Dynkin diagram. The numerical \emph{labels} of $D(A)$ are $a_0,a_1,\ldots,a_\ell$ and can be found in the aforementioned tables. The \emph{dual labels} $a_0^\vee,a_1^\vee,\ldots,a_\ell^\vee$ of $D(A)$ are the labels of $D(A^t).$ The \emph{Coxeter number} $h=\sum_{i=0}^{\ell}{a_i}$ and \emph{dual Coxeter number} $h^\vee=\sum_{i=0}^{\ell}{a_i^\vee}$ play very important roles in the theory.

Using these labels and the Cartan matrix $A=(a_{ij})_{0\leq i,j\leq \ell}$, define the following symmetric nondegenerate bilinear forms on $\mathfrak{h}$ and $\mathfrak{h}^*$,
$$(\alpha_i^\vee,\alpha_j^\vee)=a_ja_j^{\vee -1}a_{ij},\ (\alpha_k^\vee,d)=a_0\delta_{0k},  \text{ and } (d,d)=0\  (i,j,k=0,\ldots,\ell)$$
$$(\alpha_i,\alpha_j)=a_i^{\vee}a_i^{-1}a_{ij},\ (\alpha_k,\Lambda_0)=a_0^{-1}\delta_{0k},  \text{ and } (\Lambda_0,\Lambda_0)=0\ (i,j,k=0,\ldots,\ell),$$
respectively. The bilinear form on $\mathfrak{h}$ induces an isomorphism 
$$\nu:\mathfrak{h}\to \mathfrak{h}^*$$
$$\nu(\alpha_i^{\vee})=a_ia_i^{\vee -1}\alpha_i,\ \nu(d)=a_0\Lambda_0$$
which will be used to identify both vector spaces (with little to no warning). We will also be using the map $\alpha\mapsto\alpha^\vee$ where $\alpha$ is a nonisotropic element of $\mathfrak{g}^*$ and  $\alpha^\vee=\frac{2}{(\alpha,\alpha)}\nu^{-1}(\alpha).$ Define the important isotropic elements
$$\delta=\sum_{i=0}^{\ell}{a_i\alpha_i}\in \mathfrak{h}^*\text{ and } K=\sum_{i=0}^{\ell}{a_i^\vee \alpha_i^\vee}\in\mathfrak{h}.$$
The element $K$ is called the \emph{canonical central element} of $\mathfrak{g}$. Also define the element $\theta=\delta-a_0\alpha_0$, which is a dominant root of $\overset{\circ}{\mathfrak{g}}$. With these new elements define the new bases
$$\{\alpha_1,\ldots,\alpha_\ell,\Lambda_0,\delta\}\subset \mathfrak{h}^*$$
$$\{\alpha_1^\vee,\ldots,\alpha_\ell^\vee,d,K\}\subset \mathfrak{h}$$
and the projection map
$$\pi:\mathfrak{h}^*\to \overset{\circ}{\mathfrak{h}}{}^*$$
$$\alpha\mapsto\overline{\alpha}$$
with respect to the new basis. 

\bigskip  

The \emph{Weyl group} $W$ of $\mathfrak{g}$ is the group of orthogonal transformations on $\mathfrak{h}^*$ generated by the reflections
$$s_{i}(\lambda)=\lambda-\langle \lambda,\alpha_i^\vee\rangle \alpha_i$$
for $i=0,\ldots,\ell$. Using the isomorphism $\nu$, $W$ also acts by orthogonal transformations on $\mathfrak{h}$. Let $\overset{\circ}{W}$ be the subgroup generated by $s_i,$ $1\leq i\leq \ell,$ then $\overset{\circ}{W}$ is the Weyl group of the underlying simple Lie algebra $\overset{\circ}{\mathfrak{g}}$. Moreover, the Weyl group $W$ can be decomposed as the semidirect product 
$$W=\overset{\circ}{W}\ltimes T\cong \overset{\circ}{W}\ltimes M$$
where $M$ is the lattice spanned by the set $\overset{\circ}{W}(a_0^{-1}\theta)$, $T$ is the abelian group of translations by elements of $M$ (acting on $\mathfrak{h}^*$), and $\overset{\circ}{W}$ acts on $M$ in the natural way. The elements of $T$ are written as $t_\alpha,$ $\alpha\in M$. For $\mathfrak{g}$ of type $X_N^{(r)}$, in the case $r=a_0$ the lattice $M=\nu(\overset{\circ}{Q}{}^\vee)$ and in the case $r>a_0$, $M=\overset{\circ}{Q}$ (see the next subsection for the definitions of relevant lattices).

\bigskip

\subsection{Integrable Highest Weight Modules}
\label{integrable-highest}

Throughout this work, we will be dealing with the category of integrable highest weight modules of level $k$, where $k$ is a positive integer. This category is endowed with the usual direct sum structure. (Unfortunately, the usual tensor product does not preserve the level.) In order to describe these modules, let us very rapidly recall the various lattices and sets that feature in the representation theory. 

Let $\mathfrak{g}$ be an affine Kac-Moody algebra. Recall that the \emph{root lattice} $Q$ is the lattice in $\mathfrak{h}^*$ generated by the simple roots $\alpha_0,\ldots,\alpha_\ell.$ The \emph{coroot lattice} $Q^\vee$ is the lattice in $\mathfrak{h}$ generated by $\alpha_0^\vee,\ldots,\alpha_\ell^\vee.$ Define the \emph{fundamental weights} to be elements $\Lambda_i\in \mathfrak{h}^*$ such that $\Lambda_i(\alpha_j^\vee)=\delta_{i,j},$ for $i,j=0,\ldots,\ell.$ The \emph{weight lattice} $P$ is the lattice in $\mathfrak{h}^*$ generated by the fundamental weights $\Lambda_0,\ldots,\Lambda_\ell.$ The \emph{dominant integral weights} $P^+$ is the set of $\lambda \in P$ such that $\langle \lambda,\alpha_i^\vee\rangle\geq 0$ for $i=0,\ldots,\ell$. The \emph{regular dominant integral weights} $P^{++}$ are the dominant integral weights such that the above inequality is strict for all $\alpha_i^\vee$. When decorated above by the symbol $\circ$, these lattices are the corresponding lattices for $\overset{\circ}{\mathfrak{g}}$.

\bigskip

A very important element of $P^{++}$ is 
$$\rho:=\Lambda_0+\Lambda_1+\cdots +\Lambda_\ell.$$
The level $k$ of an element $\lambda\in \mathfrak{h}^*$ is the quantity $\langle \lambda,K\rangle$. Note that the level of $\rho$ is $h^\vee$. Define the sets
$$P_k=\{\lambda\in P: \langle\lambda,K\rangle = k \text{ and } (\lambda,\alpha)\in \Z\text{ for all }\alpha \in M\},$$
$$P^k=\{\lambda\in P: \langle\lambda,K\rangle = k \text{ and } \langle \lambda,\alpha^\vee\rangle\in \Z\text{ for all }\alpha \in \overset{\circ}{Q}{}^\vee\},$$
$$P_k^{+}=P_k\cap P^{+},\ P_k^{++}=P_k\cap P^{++}, \ P^{k+}=P^k\cap P^{+},\ P^{k++}=P^k\cap P^{++}.$$

\bigskip

For an affine algebra $\mathfrak{g}$, an \emph{integrable highest weight module of highest weight} $\lambda$ is a non-trivial $\mathfrak{g}$-module $V$ with decomposition
$$V=\bigoplus_{\mu\in\mathfrak{h}^*}V_\mu,$$
$$V_\mu=\{v\in V:h\cdot v=\mu(h)v\text{ for all }h\in\mathfrak{h}\},$$
such that $\lambda\in P^+$, there is an element $v_0\in V_\lambda$ (unique up to scaling) that generates $V$, and $e_iv_0=0$ for all $i=0,\ldots,\ell$. The irreducible integrable highest weight modules of highest weight $\lambda\in P^+$ are denoted $L_{\lambda}$. The central element $K$ acts by a scalar $k$, called the \emph{level}, on $L_{\lambda}$. The terminology is consistent since the level of $L_{\lambda}$ is the level of $\lambda$ as an element of $\mathfrak{h}^*$. 

\subsection{Theta Functions and Characters}
\label{theta-fns-characters}

This subsection discusses characters and theta functions associated with affine Lie algebras. For a full account of the theory, see [KP] or [Ka]. 

Let $\mathfrak{g}$ be an affine Lie algebra, $k\in \Z_{>0}$, and $\lambda \in P_k$, then the \emph{classical theta function} of level $k$ with characteristic $\overline{\lambda}$ is defined
$$\Theta_{\lambda}=e^{-\frac{(\lambda,\lambda)}{2k}\delta}\sum_{\alpha\in M}e^{t_{\alpha}(\lambda)},$$
(recall the notation $t_{\alpha}$ given at the end of \ref{affine-algebra-prelims}). It can be shown that $\lambda\equiv \mu \pmod{kM+\C\delta}$ implies $\Theta_{\lambda}=\Theta_{\mu}$. These functions are defined on the half-space $Y=\{\lambda\in \mathfrak{h}|\ \text{Re}(\lambda,\delta)>0\}$. Let $Th_k$ be the vector space spanned by $\Theta_{\lambda}$, $\lambda\in P_k$. The following result is found in [Ka,\textsection 13].

\begin{propo}
	\label{theta-basis}
	Let $k>0,$ then the set of functions $\{\Theta_{\lambda}|\lambda\in P_k\pmod{kM+\C\delta}\}$ is a $\C$-basis of $Th_k$.
\end{propo}

\bigskip

The space $Th_k$ admits a right action of the metaplectic group $Mp(2,\Z)$, which is a double cover of
$$SL(2,\Z)=\left\{\begin{pmatrix}
a & b\\c& d
\end{pmatrix}: a,b,c,d\in \Z \text{ and }ad-bc=1\right\}.$$ The group $SL(2,\Z)$ has the generators 
$$T=\begin{pmatrix}
1 & 1\\0& 1
\end{pmatrix} \text{ and } S=\begin{pmatrix}
0 & -1\\1& 0
\end{pmatrix}$$
which play a prominent role. The action of $A\in Mp(2,\Z)$ on $\Theta_{\lambda}$ is denoted $\Theta_{\lambda}|_A$. The same notation will be used when referring to the action of $SL(2,\Z)$ when applicable.

\bigskip

For $r=1,2,3,$ the congruence subgroups
$$\Gamma_1(r)=\left\{\begin{pmatrix}
a & b\\
c& d
\end{pmatrix} : a,b,c,d\in \Z \text{ and } \begin{pmatrix}
a & b\\
c& d
\end{pmatrix}\equiv \begin{pmatrix}
1 & *\\
0& 1
\end{pmatrix}\text{ (mod r)}\right\},$$
are also important. The group $\Gamma_1(r)$ is generated by the matrices
$$u_{12}=T^{-1}=\begin{pmatrix}
1 & -1\\
0 & 1
\end{pmatrix} \text{ and } 
u_{21}^r=\begin{pmatrix}
1 & 0\\
r & 1
\end{pmatrix}$$
which satisfy the relations $(u_{12}u_{21}^r)^s=1,$ where $s=3,4,6$ (resp.), and $S=u_{12}u_{21}u_{12}$ [IS].

\bigskip

If $V$ is a $\mathfrak{g}$-module with weight space decomposition 
$$V=\bigoplus_{\mu\in\mathfrak{h}^*}V_{\mu}$$
such that $\dim V_{\mu}<\infty$, the \emph{formal character} of $V$ is the formal sum
$$ch(V)=\sum_{\mu\in \mathfrak{h}^*} (\dim V_{\mu})e^{\mu}.$$
For any $\lambda\in P^{k+}$, define the \emph{normalized character} or simply \emph{character}
$$\chi_{\lambda}=e^{-m_{\lambda}\delta}ch(L_\lambda)$$
where $m_{\lambda}$ is the so-called \emph{modular anomaly} of $\lambda$
$$m_{\lambda}=\frac{|\lambda+\rho|^2}{2(k+h^\vee)}-\frac{|\rho|^2}{2h^\vee}.$$
The complex vector space spanned by the characters $\chi_\lambda$ of $\mathfrak{g}$ of level $k$ will be denoted $Ch_k:=Ch_k(\mathfrak{g})$. The characters can be expressed in terms of classical theta functions. To explain this, first define the \emph{alternants}
$$A_{\lambda}=\sum_{w\in \overset{\circ}{W}}\epsilon(w)\Theta_{w(\lambda)}$$
and let $Th_k^{-}$ be the space spanned by the $A_{\lambda}$ for $\lambda\in P^{k+}$. The following result can also be found in [Ka,\textsection 12].

\begin{propo}
	If $k>0$ and $\lambda\in P^{k+}$, then
	$$\chi_{\lambda}=A_{\lambda+\rho}/A_{\rho}.$$
	Furthermore, $\chi_{\lambda}\in Th_{k}$. 
\end{propo}

\bigskip

In the case when $r\leq a_0$, the space of characters is invariant under the action of $SL(2,\Z)$, i.e., the action of $Mp(2,\Z)$ factors through $SL(2,\Z)$. In the case $r>a_0$, $Ch_k$ is invariant under certain subgroups of finite index in $SL(2,\Z)$. The explicit actions of these subgroups will be discussed in the upcoming sections.

\end{document}